\documentclass[12pt,oneside,reqno]{amsart}     % use "amsart" instead of "article" for AMSLaTeX format
\usepackage{geometry}
\usepackage{float}
\allowdisplaybreaks
\usepackage{graphicx}               % Use pdf, png, jpg, or eps with pdflatex; use eps in DVI mode
\usepackage{amssymb}
\usepackage{amsmath}
\usepackage[all]{xy}
\usepackage{mathtools}
\usepackage{tikz-cd}
\usepackage{enumitem}
\usepackage{tikz-cd}
\usepackage{appendix}
\usepackage[T1]{fontenc}
\usepackage{esint}

\usepackage[%
bookmarks=true,
colorlinks,
linkcolor=blue,
urlcolor=blue,
citecolor=blue,
plainpages=false,
pdfpagelabels,
final,
breaklinks=true\between
]{hyperref}
\usepackage{hyperref}
\usepackage[numbers,sort&compress]{natbib}

\newtheorem{thm}{Theorem}[section]
\newtheorem{lem}[thm]{Lemma}
\newtheorem{cor}[thm]{Corollary}
\newtheorem{prop}[thm]{Proposition}
\newtheorem{conj}{Conjecture}[section]

\newtheorem{rem}{Remark}[section]

\newtheorem{defn}{Definition}[section]

\numberwithin{equation}{section}

\newcommand{\R}{\mathbb{R}}

\def\Pb{\ifmmode{\Bbb P}\else{$\Bbb P$}\fi}
\def\Z{\ifmmode{\Bbb Z}\else{$\Bbb Z$}\fi}
\def\C{\ifmmode{\Bbb C}\else{$\Bbb C$}\fi}
\def\S{\ifmmode{S^2}\else{$S^2$}\fi}

\def\supp{\operatorname{supp}}

\def\i{\mathrm{int}}

\def\S{\cal S}

\definecolor{OklahomaRed}{HTML}{841617}%AM comment
\definecolor{OklahomaRed}{HTML}{841617}%AM edit

\definecolor{KTHBlue}{HTML}{004791}%LY comment
\definecolor{KTHBlue}{HTML}{004791}%LY edit

\newcommand{\AxisRotator}[1][rotate=0]{%
    \tikz [x=0.25cm,y=0.60cm,line width=.2ex,-stealth,#1] \draw (0,0) arc (-150:150:1 and 1);%f
}

\begin{document}
\title{An Ancient Stacked Pancake Solution to Mean Curvature Flow}

\begin{abstract} 
    We construct an embedded ancient solution to the mean curvature flow which is qualitatively given by a ``stack'' of two ancient pancakes joined by a neck. Our solution is closed, non-convex, and contained in a slab.
\end{abstract} 

\author{Alexander Mramor and Louis Yudowitz}
\address{Department of Mathematics, University of Oklahoma, Norman, OK 73019, USA
\newline
\newline
\indent Department of Mathematics, KTH Royal Institute of Technology, Lindstedtsv\"agen 25, 114 28 Stockholm, Sweden}
\date{}
\email{amramor@ou.edu, yudowitz@kth.se}

\maketitle

\section{Introduction}\label{intro}

    Ancient solutions to mean curvature flow (MCF)
    are flows which exist for all negative times. Such special solutions are of central importance in the singularity analysis of MCF, as well as being natural analogues of complete ancient solutions to the heat equation in submanifold geometry. They have been intensely studied under various geometric and topological restrictions and many have been constructed. An interesting recent family of examples of such flows are the ``ancient pancakes'' constructed in \cite{blt_pancakes}. These examples are ancient flows of convex hypersurfaces that are rotationally symmetric and trapped in a fixed slab region (which is orthogonal to the axis of rotation). They are also symmetric under reflection about a hyperplane which is parallel to the slab.

In this article, we construct a new ancient solution to the mean curvature flow which qualitatively corresponds to two parallel copies of an ancient pancake solution joined by a neck. Stated more precisely, our main result is:

    \begin{thm}\label{mainthm} 
        For every $n \geq 3$ there exists an ancient stack of two pancakes. That is, there exists an ancient solution of MCF $M^n_t \subset \mathbb{R}^{n+1}$ defined for all $t \in (-\infty, 0]$ with the following properties:

        \begin{enumerate}
        \item $M^n_t$ is a closed hypersurface and is contained in a slab.
            \item $M^n_t$ is $O(n)\times O(1)$ invariant (rotationally and reflection symmetric), embedded, and non-convex. 
            \item Moreover, the profile curve $\Gamma_t$ of $M^n_t$ is graphical over the axis of rotation and at any time $t \leq 0$, $\Gamma_t$ has $2$ local maxima and $1$ local minimum.

        \end{enumerate}
    \end{thm}

\begin{figure}\centering
\includegraphics[width=0.85\textwidth]{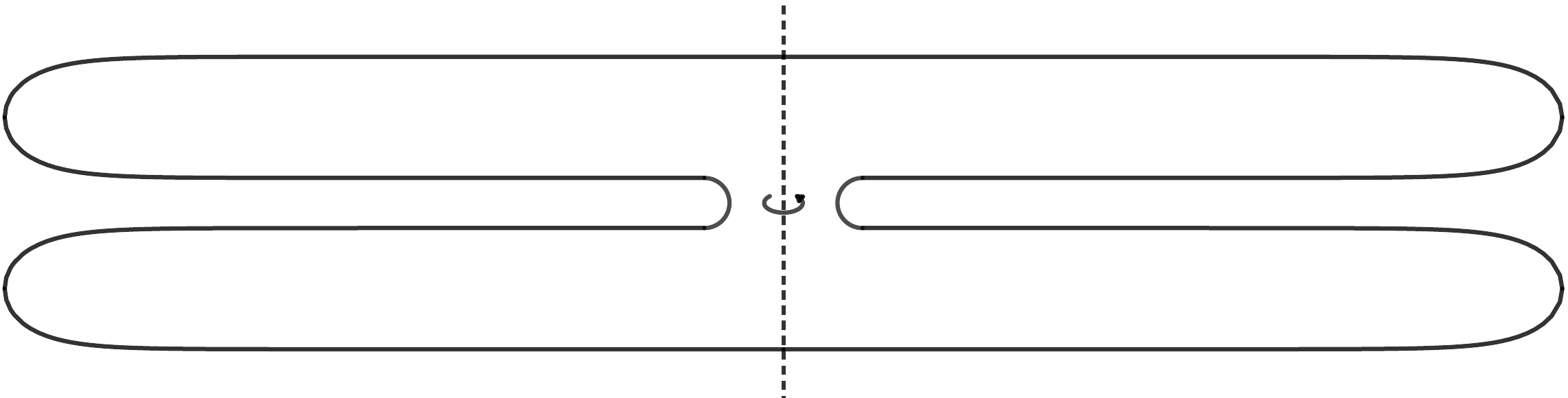}
\caption{A ``stack'' of two ancient pancakes.}\label{fig:pancake stack}
\end{figure}

    Informally, the local maxima mentioned above correspond to ``tips'' of the pancakes and the local minimum corresponds to a ``neck'' joining them. The construction will be reflection symmetric across the hyperplane perpendicular to the rotation axis which passes through the local minimum; throughout the article we use these symmetries, particularly the rotational symmetry, quite strongly. Ideally, one can construct similar ancient flows with minimal (or no) symmetry, say in the spirit of Traizet \cite{traizet} for minimal surfaces --- for instance, perhaps one may stack ancient polygons \cite{blt_poly}.

    The inspiration for this construction comes from the observation that, for $n \geq 3$, the minimal catenoid lies in a slab and is asymptotic to the sides of said slab. Moreover, the ancient pancakes are also asymptotic to the sides of the smallest slab they lie in as $t \to -\infty$. It therefore seems reasonable that in suitable dimensions ancient pancake stackings can be constructed by first joining them together using catenoids to create approximate ancient solutions, and then applying a Newton method type argument similar to the gluing procedure for minimal surfaces.    

    Such a strategy has been successfully carried out in the parabolic setting in previous works; consider, for instance, \cite{BreKap}. However, here we proceed by using a continuity argument, closer in spirit to the rigorous construction of Hamilton's degenerate dumbbell in \cite{aag} or the more recent construction of genus 1 self shrinkers in \cite{Chu2023GenusoneSI}. This method gives less explicit information about the solution constructed, though. For instance, it is not clear that the backwards limit of our ancient solution is, outside some large parabolic neighborhood $U \times (-T, 0]$, very close to a disjoint union of two parallel ancient pancakes. On the other hand, our approach is quicker and less technical than many gluing constructions. 
        
    The assumption that $n \geq 3$, compared to the ansatz given above, enters the proof in a somewhat roundabout way through an intersection point argument using that higher dimensional catenoids lie in slabs of finite thickness in Lemma \ref{smooth_seq}. This argument does not seem to extend to the case $n = 2$. This failure at low dimension seems to be natural both with the ansatz in mind and that, by a further intersection point argument, one can see that the necks of the ancient stacks must be trapped in a bounded region for all times. This suggests that the limiting profile of the neck, if smooth and well defined, should be a catenoid as $t \to -\infty$, so that the backwards limit of the ancient flows are disjoint unions of planes and catenoids which is ruled out by the halfspace theorem in $\mathbb{R}^3$. Alternately, perhaps in the $n = 2$ case one can instead construct an ancient solution where for $t \ll 0$ the neck flows out of a ``cusp'' along the axis of rotation corresponding to in a sense a singularity at $t = -\infty$ and, instead of proceeding as in Lemma \ref{smooth_seq}, argue that the limiting flow must be an almost regular flow using a modification of Theorem 1.7 in \cite{BK} which in fact would suffice for our purposes. Related phenomena (considering we use time recenterings of approximating flows) was explored in \cite{SunChen_mult2plane}, where the authors construct a rotationally symmetric mean curvature flow which converges to a plane with multiplicity 2 in the long term. Although we do not seriously study the backwards asymptotics of our ancient solution here, understanding them more clearly may be useful in answering more refined uniqueness questions discussed in \cite{Mra1}. 
   
    Our solution is interesting for a number of reasons. Most basically, it gives a new example of an ancient solution trapped in a slab and seems to be the first example of a non-convex compactly embedded one. As such, it shows that Theorem 1.2 in \cite{blt_pancakes} fails without the convexity hypothesis, at least when $n\ge 3$. 
    
    Another motivation comes from the following result: in \cite{chini_moller}, Chini and M{\o}ller showed that the convex hull of the spacetime track of a properly immersed mean curvature flow in $\mathbb{R}^{n+1}$ must either be all of $\mathbb{R}^{n+1}$ (that is, entire), a slab (collapsed), or a halfspace $\mathbb{R}_{+}^{n+1}$. In particular, the convex hull may not be a wedge. There is the question, of course, of whether or not this result is sharp. Currently, the only example of an ancient flow corresponding to a halfspace above, due to John Man Shun Ma in \cite{ma_immersed_halfspace}, is an immersed curve shortening flow. A natural question is if our construction can be in some sense iterated to produce more (embedded) examples of ancient flows whose convex hulls sweep out a halfspace, corresponding to what we'll refer to as a "half-infinite" stack of ancient pancakes. Because of its relevance we highlight it, stated optimistically: 
    \begin{conj}\label{stack_conj} For each $n \geq 3$, there exists an ancient solution $M^n_t \subset \R^{n+1}$ whose convex hull sweeps out a halfspace corresponding to a half-infinite stack of pancakes. 
    \end{conj} 
    
    In the context of our methods what seems to be the most obvious idea towards showing this is to first try to directly extend our construction to a stacking of $2 < k < \infty$ pancakes, perhaps using induction at the step of creating the old-but-not-ancient approximating solutions (see section \ref{old_not_ancient_construction}). Supposing such a construction can be carried out, by taking a limit of these ancient solutions as $k \to \infty$ it seems probable, given the estimates one has at hand, that one could obtain a smoothly embedded ancient flow of $\mathbb{R}^{n+1}$ whose convex hull sweeps out exactly a halfspace in all dimensions $n \geq 3$ as long as one has uniform $C^0$ bounds on the profile curves (as graphs over the axis of rotation) at any given time.

    A related, but less direct, approach is the following. First, it seems we should be able to apply our methods to construct an infinite, simply periodic stack of ancient pancakes, one whose convex hull is entire; using periodicity there are some steps which seem more straightforward than in the finite stacking case. With this in hand, by approximating such a flow with sensibly designed compact ones and recentering them appropriately it seems plausible that one can construct a sequence of old-but-not-ancient flows which limit to a half-infinite stack of pancakes, fulfilling the conjecture above. 

    \textbf{Acknowledgments:} During his time at Copenhagen, when some of this work was completed, AM was supported by CPH-GEOTOP-DNRF151 from the Danish National Research Foundation (via the GeoTop centre) and is grateful for their assistance. LY was supported by foundations managed by the Royal Swedish Academy of Sciences, as well as the G\"oran Gustafsson foundation. The authors also sincerely thank Mat Langford for many helpful conversations through the course of this project as well as Paco Mart\'{\i}n for his interest and feedback and  Alec Payne for his helpful comments, particularly concerning Lemma \ref{continuous_dependence}.
\section{Preliminaries}\label{prelims}
    In this section we give a non-exhaustive account of some of the results and notions we will use in this paper, primarily focusing on those concerning the flow. We start with the most basic definition:

    \begin{defn} 
        A (smooth) mean curvature flow of embedded hypersurfaces in $\mathbb{R}^{n+1}$ is given by a manifold $M^n$ and a family $F: M \times I \to \mathbb{R}^{n+1}$ of embeddings satisfying
        \begin{equation}
            \frac{\partial F}{\partial t}\left(x,t\right) = -H \nu,
        \end{equation}
        where $I \subseteq \mathbb{R}$ is some nonempty interval, $H$ is the mean curvature of $F$, and $\nu$ is the unit normal vector.
    \end{defn}

    We shall denote the image of $F$ at time $t$ by $M_t$.

\subsection{Weak solutions}

We shall also require the following weak formulation(s) of the flow, starting with the most well known GMT version: 

    \begin{defn}[Brakke flow]
        A (n-dimensional integral) Brakke flow is a family of Radon measures $\mu_t$ such that, on an interval $I \subseteq \mathbb{R}$:

        \begin{enumerate}
            \item For almost every $t \in I$ there exists an integral $n$-dimensional varifold $V(t)$ with $\mu_t = \mu_{V(t)}$ so that $V(t)$ has locally bounded first variation and has mean curvature vector $\vec{H}$ orthogonal to Tan$(V(t), \cdot)$ a.e. 
            \item For a bounded interval $[t_1, t_2] \subset I$ and any compact set $K \subset \mathbb{R}^{n+1}$,
            \begin{equation}
                \int_{t_1}^{t_2} \int_K (1 + H^2) d\mu_t dt < \infty.
            \end{equation}
            \item (Brakke's inequality) If $I \subset [t_1, t_2]$ and $\phi \in C^1_c(\mathbb{R}^{n+1} \times [t_1, t_2])$ with $\phi \geq 0$, then
            \begin{equation}
                \int_{V(t_2)} \phi(\cdot, t_2) d\mu_{t_2}  \leq \int_{V(t_1)} \phi(\cdot, t_1) d\mu_{t_1} + \int_{t_1}^{t_2} \int_{V(t)} -\phi H^2 + H \langle \nabla \phi, \nu \rangle + \frac{d\phi}{dt} d\mu_t dt.
            \end{equation} 
        \end{enumerate}
    \end{defn}

    The study of Brakke flows itself is quite rich, and a comprehensive introduction to Brakke flows can be found in, say, \cite{Ton, schulze_grenoble_notes}. Here, we will briefly and informally discuss a couple results which will be needed about them. First is \textbf{Brakke's compactness theorem}, which says that, for a sequence of Brakke flows with uniform area bounds on parabolic cylinders, one may exact a subsequence which converges to another Brakke flow. This convergence will be in the sense of Radon measures for all times and as varifolds for almost all times and if each Brakke flow in the sequence is unit regular and cyclic the limit will be as well. For the area bounds it suffices to have a uniform bound on the  Colding-Minicozzi entropy $\lambda(\Sigma)$, developed in \cite{CMentropy}. A second important result is \textbf{Brakke's regularity theorem}, which says that Brakke flows with density bounds sufficiently close to 1 in a backwards parabolic neighborhood will be smooth with bounded curvature in a smaller neighborhood. Although we are mainly concerned with smooth flows in this paper, our assumptions generally will not be strong enough to employ, for instance, Arzel\'a--Ascoli (in particular, a compactness theorem in the $C^k_{loc}$ topology), which necessitates working in this wider class of solutions. As is well known, there are a number of ways to make sense of weak solutions for the flow. Another one, which captures, in essence, uniqueness properties, will also play a role for us:

    \begin{defn}[Weak set flow]\label{original_defn} 
        Let $W$ be an open subset of a Riemannian manifold and consider $K \subset W$. Let $\{\ell_t\}_{t \geq 0}$ be a one -parameter family of closed sets with initial condition $\ell_0 = K$ such that the space-time track $\cup \ell_t \subset W$ is relatively closed in $W$. Then $\{\ell_t\}_{t \geq 0}$ is a weak set flow for $K$ if, for every smooth closed surface $\Sigma \subset W$ disjoint from $K$ and serving as the initial data for a smooth MCF defined on $[a,b]$, we have     
        \begin{equation}
            \ell_a \cap \Sigma_a = \emptyset \implies \ell_t \cap \Sigma_t = \emptyset
        \end{equation} 
        for each $t \in [a,b]$.
    \end{defn}

The set theoretic level set flow is the largest weak set flow. 

    \begin{defn}[Level set flow] 
        The level set flow of a set $K \subset W$, which we denote by $K_t$, is the maximal weak set flow. That is, a one-parameter family of closed sets $K_t$ with $K_0 = K$ such that if a weak set flow $\ell_t$ satisfies $\ell_0 = K$ then $\ell_t \subset K_t$ for each $t \geq 0$. The existence of a maximal weak set flow is verified by taking the closure of the union of all weak set flows with a given initial data. 
    \end{defn}

    An interesting phenomenon which may occur is called \textbf{fattening}. This is when the level set flow develops a nonempty interior and can be seen as a measure of uniqueness. When, for instance, the flow of a hypersurface is non-fattening and has only finitely many singularities there can only be one ``well behaved'' (unit regular, cyclic mod $2$) Brakke flow. When the initial data is compact and smoothly embedded, nonfattening is guaranteed generically, as discussed in \cite{I1}, and also importantly for us is known to hold for \textit{all} smoothly embedded, compact, rotationally symmetric initial data with graphical profile function. In fact, as we'll discuss below, these only encounter finitely many singularities. 

    \subsection{Rotationally symmetric flows}

To set notation, we consider a (say, compact) rotationally symmetric initial hypersurface $M_0$ given by the revolution of its profile curve $\Gamma_0$ about the $x := x_{n+1}$-axis. We denote the other variable by $r$. Taking $\Gamma_0$ to lie in the $xr$-plane $P$, rotational symmetry is preserved under the flow and the corresponding evolution of $\Gamma_t$ of $\Gamma_0$ in $P$ is given, in terms of a family $\gamma:S^1\times [0,T)\to P$ of parametrization maps, by
    \begin{equation}\label{rotsym_evol}
        \frac{\partial\gamma}{\partial t} = -\left(\kappa+\frac{n-1}{r}\cos\theta\right)\nu,
    \end{equation}
    where here $\kappa(\cdot,t)$ is the geodesic curvature of $\gamma(\cdot,t)$ in the $xr$-plane and $\theta$ is the angle that the normal makes with the perpendicular of the axis of rotation. We therefore see that the MCF of rotationally symmetric surfaces is essentially the curve shortening flow, which is very well understood, with an additional forcing term that decays away from the axis of rotation. Owing to its simplicity, the flow \eqref{rotsym_evol} is particularly well behaved, and we summarize some important results on it that were proved in the important paper \cite{aag} of Angenent, Altschuler, and Giga (with some differences in notation). In the following statements, as in \cite{aag}, we make the additional assumption that $\Gamma_0$ is graphical over the $x$-axis. This quality is also preserved under the flow between singular times. Writing the graph function as $u(x,t)$, $u$ satisfies
    \begin{equation}\label{rotsym_graph_evol}
        \frac{\partial u}{\partial t} = \frac{\partial^2_{x}u}{1 + \left(\partial_x u\right)^2} - \frac{n-1}{u}.
    \end{equation}
    That this quality is preserved (Theorem 4.3(a) in \cite{aag}) is a consequence of a strong version of the Sturmian principle using planes perpendicular to the axis of rotation. By Sturmian principle, we mean the following: 

    \begin{thm}[Theorem 4.1 in \cite{aag}]\label{aag_thm0}
    Given two rotationally symmetric mean curvature flows defined on a common time interval $[0,T)$, the respective profile curves either coincide for all $t \in (0, T)$, or exhibit finitely many intersections for all $t \in (0,T)$. In the second case, the number of intersections is non-increasing in time, and decreases whenever a nontransverse intersection occurs. 
    \end{thm}
    
     As pointed out in \cite{aag}, the Sturmian principle also holds for noncompact flows so long as their intersection set as a subset of spacetime is compact (which for us will always be the case). The following is another consequence of Sturmian theory for \eqref{rotsym_evol}; this and the next few statements are paraphrasings of parts of Theorems 1.1, 1.2 in \cite{aag}:

    \begin{thm}\label{aag_thm2}
        For compact, rotationally symmetric initial data $M_0$ with graphical profile function $\Gamma_0$ as indicated above, the number of critical points of $\Gamma_t$, in terms of distance to axis of rotation, is non-increasing along the flow. 
    \end{thm}
    
    In the statement above, local minima correspond to ``necks'' and the content of Theorem \ref{aag_thm2} can be interpreted as saying these can pop out (or the local maxima can rush in), but no new necks can form. As is well known, singularities are ubiquitous in the flow; the following concerns these:

    \begin{thm}\label{aag_thm3} 
        With $M_0$ as above, the singularities of $M_t$ will be mean convex and occur exactly along the axis of rotation. The number of singularities (and hence singular times) is bounded from above by the number of critical points of $\Gamma_0$.
    \end{thm}

    The picture here is that typically singularities corresponding to local minima will be modeled (in terms of tangent flows) on shrinking cylinders $S^{n-1} \times \mathbb{R}$, while those corresponding to local maxima will be modeled on shrinking spheres. Of course, after singularities the flow may become disconnected, in particular when neckpinches occur, but when the flow is nonempty we will still refer to the profile curve as $\Gamma_t$. As already suggested by the notation above, the last statement from \cite{aag} we give says essentially that the weak flows through such singularities are well defined:

    \begin{thm}\label{aag_thm1}
        For compact, rotationally symmetric initial data $M_0$ with graphical profile function $\Gamma_0$, the level set flow $M_t$ starting from $M_0$ is non-fattening and remains rotationally symmetric. The flow will be smooth away form the axis of rotation. 
    \end{thm}

    The following result gives good gradient estimates starting from bounded initial data which will be useful in the sequel. However we see that the bounds below depend on the height of the profile curve over the axis of rotation, which in our construction of old-but-not-ancient solutions we will take arbitrarily large:

\begin{thm}[Theorem 4.3b in \cite{aag}]\label{aag_grad} Let $M_t$, $0 \leq t < T$, be a family of smooth hypersurfaces evolving by their mean curvature and assume that it is given by a smooth solution of the horizontal graph equation $r = u(x,t)$, with $0 < t < T$ and $a(t) < x < b(t)$ for certain smooth functions $a, b: (0, T) \to \R$. Then there is a function $\sigma: \R_+ \times \R_+ \to \R$ such that 
\begin{equation}
|u_x(x,t)| \leq \sigma(t, u(x,t))
\end{equation}
holds for all $0 < t < T$, $a(t) < x < b(t)$. The function $\sigma$ only depends on $\sup u(x,0)$.
\end{thm}

              The  proof is an application of the Sturmian principle using a cleverly designed comparison flow/profile curve. As pointed out in \cite{aag} without loss of generality $\sigma$ is monotonically nonincreasing in $u$ and $t$. Note by the Ecker--Huisken estimates \cite{EH} the above implies analogous higher order bounds for the rotated surface away from the axis of rotation, although they necessarily degenerate as one approaches it. On the other hand, as long as we have uniform curvature bounds on the initial data, as in section \ref{old_not_ancient_construction} (a step where these estimates are prominently applied), we will have uniform curvature bounds for small times by the doubling time estimate, before the estimates above strongly apply. 
       
       In section \ref{old_not_ancient_construction}, we will also wish to potentially apply the estimates above past singular times, which we now discuss. Now, even through singular times away from the origin the flow is smooth and satisfies the classical mean curvature flow, and the profile curve of $M_t$, $\Gamma_t$ (in our notation), will continue to be locally represented by a graph of a function, at smooth times one for each of its connected components. One can see the number of intersection points of $\Gamma_t$ with the comparison profile curve used in the proof will not increase through a singular time. From this, again keeping in mind the argument is via the Sturmian principle, one can see that the estimates also apply through singular times for each of the connected components of $\Gamma_t$, maintaining the monotonicity of $\sigma$. Alternately, one can use local estimates \cite{chen2007uniqueness, EH} away from the axis of rotation about a singular time along with "restarting" the estimates above to get control on the gradient away from the axis of rotation -- some care is needed to ensure such estimates one derives only depend on $C^0$ bounds but this is possible at least in the specific setting we will apply them in, where there is at most one singular time (on the time domain we restrict ourselves to) and all initial data has uniformly bounded curvature.

     In the sequel we will be applying the Brakke compactness theorem to smooth rotationally symmetric flows, and it will be helpful to go between them and their limit. The following lemma relates this convergence to Hausdorff convergence; this is similar to a well known general fact for Hausdorff convergence of spacetime supports: 
    
    \begin{lem}\label{Hau_conv} Suppose that $U \subset \mathbb{R}^2$ is a bounded open set whose closure is disjoint from the axis of rotation. Let $M^i_t$ be a sequence of smooth rotationally symmetric mean curvature flows which converge to a limiting Brakke flow $B_t$ (as Radon measures for all times and varifolds for almost a.e. time, as in the Brakke compactness theorem). Moreover, assume the profile curves of $M^i_t$, denoted by $\gamma^i_t$, are graphical over the axis of rotation. Then the $M^i_t$ converge in the Hausdorff topology to $\supp(B_t)$ in $U'$, the set obtained by rotating $U$ around the axis of rotation.
\end{lem}
\begin{proof}
Since $\overline{U}$ is disjoint from the axis of rotation $R$ and is compact, we see that trivially there exists a uniform choice of $\epsilon > 0$ so that for each $x \in \overline{U}$ $B(x, \epsilon)$ is disjoint from $R$. Considering a fixed flow and time let $x \in \gamma^i_t$. Since the flows $M^i_t$ are closed with graphical profile curves $\gamma^i_t$, we observe that we can bound the mass of $\gamma^i_{t} \cap B(x, \epsilon/2)$ from below by $\epsilon/2$. This follows from considering that the profile curves $\gamma^i_t$ must enter and exit the ball at least once and they pass through the center $x$. Since $\partial B(x, \epsilon/2)$ is at least $\epsilon/2$ distance from the origin, we can bound the mass of $M^i_t$ from below by $C(n)\epsilon^n$ in the ball $B(\overline{x}, \epsilon/2)$ where $C(n) > 0$ is a constant depending only on the dimension $n$. Here $\overline{x}$ is any point corresponding to $x$ under the rotation of the profile curve. Similar reasoning applies for any $0< \epsilon' < \epsilon$.

We claim this mass lower bound implies that the flows $M^i_t \cap \overline{U'}$ must converge for each $t$ in the Hausdorff topology to $\supp(B_t)$ in $\overline{U'}$. In the following, recall that the set of compact subsets of a compact set are compact in the Hausdorff topology (i.e. with respect to Hausdorff convergence). We show in $\overline{U}$ that the Hausdorff limit $M^\infty_t$ of $M^i_t$ exists from which it quickly follows by similar reasoning that $M^\infty_t$ must coincide with $\supp(B_t)$. If the Hausdorff limit did not exist, then from the preceding remark there exists two different subsequences $i_j, i_k$ for which $M^{i_j}_t \cap \overline{U'}, M^{i_k}_t \cap \overline{U'}$ converge to sets $H^1 \neq H^2 \subset \overline{U'}$. As detailed before, we have a uniform lower bound for the mass in balls, so each point in $H^1$ and $H^2$ must lie in $\supp(B_t)$, giving that $H^1, H^2 \subset \supp(B_t) \cap \overline{U'}$. Note that this relies on the mass lower bound being uniform in $i$ and $x \in U$, and likewise for the corresponding bound for $0 < \epsilon' < \epsilon$. On the other hand because $M^i_t$ converges to $B_t$ for each $t$ as Radon measures, we see that $\supp(B_t) \cap \overline{U'} \subset H_1$ and $\supp(B_t) \cap \overline{U'} \subset H_2$. From the previous step this yields a contradiction to $H_1 \neq H_2$, and thus proves our claim. \end{proof}

  We will also occasionally need to refer to ``caps'' of the evolving profile curves. These of course correspond to where the profile curve meets the axis of rotation as in Figure \ref{fig:pancake stack}, but to be more precise they are defined as follows.

    \begin{defn}
        The \textbf{upper and lower caps} of a curve $\Gamma_t$ evolving by \eqref{rotsym_graph_evol} are the components of
        \begin{equation*}
            C_\alpha := \left\{\left(x,r\right) : 0 \leq r \leq \alpha, x = v_i\left(r, t\right),~~i=1,2\right\},
        \end{equation*}
        where $\alpha > 0$ and $v_1,v_2$ are, respectively, the inverses of the graph function $u$ on the intervals $\left(a_1, a_1 + \delta\right]$ and $\left[a_2 - \delta, a_2\right)$. Here $a_1,a_2$ are the $x$-coordinates of the end points of the curve and $0 < \delta $ is a sufficiently small constant.
    \end{defn}

\section{Construction of the Ancient Solution}\label{construction}
    The construction begins by first finding a sequence $M^i_t$ of ``old-but-not-ancient'' rotationally symmetric solutions defined on time intervals $[-T_i, 0]$ for a sequence of times $T_i \to \infty$ as $i \rightarrow \infty$. Perhaps the most straightforward idea, and one which has  found success in the past (for instance, in \cite{MP, blm}), is to take a convergent subsequence of these flows \'a la Arzel\'a--Ascoli to obtain an ancient flow. To this end one would check the following:

    \begin{enumerate}
        \item Along the flows $M^i_t$ we have uniform curvature bounds, at least on domains $B(0,R) \times [-T_i, 0]$.
        \item The flows $M^i_t$ do not degenerate as $i \to \infty$ in the sense that we do not obtain in the limit an ancient flow which is trivial or has been constructed in other works. 
    \end{enumerate}

    In our setting we do in fact have helpful curvature bounds we make use of several times, by way of Theorem \ref{aag_grad} and the Ecker-Huisken estimates \cite{EH}. They nearly fit the requirements for (1) but aside from necessarily degenerating about the axis of rotation they depend on apriori $C^0$ bounds which we don't have, at least immediately, to apply as in the scheme above  -- see the proof of Lemma \ref{smooth_seq} below. However, in our setting we will at least be able to quickly apply Brakke's compactness theorem. So, instead of following the scheme above, we first construct appropriate old-but-not-ancient flows $M^i_t$ via a continuity argument and take a converging subsequence of these as Brakke flows to get a ``weak'' ancient solution $B_t$. Afterwards, we show that the weak solution is bounded for every fixed time and smooth, which hinges on using the tightly constrained nature of rotationally symmetric flows. 
    
    Indeed, as one finds below our arguments for these properties are mostly in terms of the approximating flows, so a natural question might be why actually taking a limit as Brakke flows is useful. We do not mean to claim from the above that taking such a limit is unavoidable, but having the ancient Brakke flow $B_t$ at hand does seem genuinely useful at least at some sort of organizational level. This is perhaps most clear in reducing to the ``cusps'' case in the proof of Lemma \ref{silky_smooth} below. There, constraints on the ancient Brakke flow help shed light more  precisely on how the approximating flows can degenerate in a particular case.

    \subsection{Construction of smooth ``old--but--not--ancient'' approximating flows and the ancient Brakke flow $B_t$.}\label{old_not_ancient_construction}
    We now turn to discussing the construction of the initial data for the old-but-not-ancient flows. Perhaps of some interest is that apparently we need only specify relatively coarse information about them. 
    
    Denote by $h(t)$ and $\ell(t)$ the ``thickness'' (total horizontal displacement along the axis of rotation) and ``radius'' (maximum distance from the axis of rotation) of the time $t$ slice of the standard ancient pancake solution $\{P_t\}_{t\in(-\infty,0)}$. For more precise definitions, we refer the reader to \cite{blt_pancakes}. Then, by Theorem 1.1 of \cite{blt_pancakes} we have:
    \[
    h(t)=\pi-o(1)\;\;\text{and}\;\;\ell(t)=-t+(n-1)\log(-t)+C_n+o(1)\;\;\text{as}\;\; t\to-\infty\,.
    \]

    Denote by $\gamma(\cdot,t)$ the profile curve of $P_t$, and recall we denote the axis of rotation by the $x$-axis and the other variable by $r$. We then construct the profile curves used to find our old-but-not-ancient solutions as follows. First, for some choice of $s <0$ take two copies $\gamma^u(\cdot,s), \gamma^\ell(\cdot,s)$ (for upper and lower) of $\gamma(\cdot,s)$ and arrange them so that they correspond to two parallel pancakes of distance 2 apart. That is, the lower intersection point of $\gamma^u(\cdot,s) \cap \{(x,r) \mid r = 0 \}$ is $x = 1$ and the upper intersection point of $\gamma^\ell(\cdot,s)\cap \{(x,r) \mid r = 0 \}$ is $x = -1$.  Then, for each $\rho\in [\rho_0(s),\ell(s)]$, we may join these two pancakes together with a ``neck''. This is obtained by removing the region $\{(x,r):-1-\frac{h(s)}{2}<x<1+\frac{h(s)}{2}, 0\le r<\rho\}$ and attaching the circular arc which meets the boundary points tangentially, with $\rho_0(s)$ being the value for which the circular arc touches the axis of rotation at the origin. We may mollify about the transition region to ensure the resulting profile curve (and corresponding hypersurface) is smooth (for $\rho > \rho_0(s)$) and graphical over the rotation axis.

    The distance $m^0$ from the axis to the neck will be monotone increasing with respect to $\rho$ and corresponds to a (potentially weak) local minimum of the distance to the axis of rotation. This circular arc is graphical over the axis of rotation, and we denote it by $f_{m^0}$, where $m^0$ is the local minimum just described. Thus, in this way, we may produce a profile curve $\Gamma^{m^0}_{s}$ for each $m^0\in [\delta,\ell(s)]$ where we take $\delta \ll 1$, potentially adjusted below. We highlight some especially important properties for the sequel:

    \begin{enumerate}
        \item $\Gamma^{m^0}_s$ is graphical over the $x$-axis and has two local maxima of values/height $M_1, M_2$ corresponding to those of the pancakes, and one local minimum $m^0$ corresponding to the minimum of $f_{m^0}$ which we call the \textbf{initial neck parameter}. We can (and will) take the configuration to be reflection symmetric.
        \item The compact domains bounded by $\gamma^u_s, \gamma^\ell_s$ and the $x$-axis lie in the domain bounded by $\Gamma^{m^0}_s$. 
        \item The family $\Gamma^{m^0}_s$ is continuous in $m^0$ and, in the same sense as the item above, if ${m^0_1} < {m^0_2}$, then $\Gamma^{m^0_1}_s \subset \Gamma^{m^0_2}_s$. 
    \end{enumerate}

    Refer to Figure \ref{profile_curve_fig} for a illustration of $\Gamma^{m^0}_s$. The idea is to now construct a sequence of old-but-not-ancient solutions by varying the parameter $m^0$ appropriately along a sequence $s_i \to -\infty$. %First, however, we need some technical lemmas.

   Since rotational (as well as reflection) symmetry is preserved along the flow, we often consider the behavior of the profile curve $\Gamma_s^{m^0}$ along the flow \eqref{rotsym_evol}. We denote the hypersurface obtained from $\Gamma_s^{m^0}$ by $\Sigma_s^{m^0}$ and its corresponding MCF by $(\Sigma_s^{m^0})_t$ and analogously denote the profile curve of $(\Sigma_s^{m^0})_t$ by $\left(\Gamma^{m^0}_s\right)_t$. Of course, the flow of such a hypersurface of rotation may develop singularities, and Theorem \ref{aag_thm3} tells us that these will occur precisely when a critical point of $(\Gamma^{m^0}_s)_t$ meets the $x$-axis. In particular, past the first singular time we may only discuss the flow in terms of weak solutions. However, as discussed in Section \ref{prelims} (in particular, Theorem \ref{aag_thm1}), the level set flow of rotationally symmetric compact initial data is non-fattening and remains rotationally symmetric, implying there is a unique ``well behaved'' weak flow which comes from $\Sigma_s^{m^0}$. This tells us that even through singularities it makes sense to consider the profile curves of the weak flow, in the sense that they are well defined sets of locally finite $\mathcal{H}^1$ measure whose orbit under rotation gives $(\Sigma_s^{m^0})_t$. We shall denote the weak flow of $\Sigma_s^{m^0}$ by $\{(\Sigma_s^{m^0})_t\}_{t\in[0,\omega_{m}(s))}$, and the weak flow of its profile curves by $\{(\Gamma_s^{m^0})_t\}_{t\in[0,\omega_{m^0}(s))}$. Here $\omega_{m^0}(s)$ denotes the time the weak flow goes extinct.

   Since the pancakes $\gamma^u(\cdot, s), \gamma^\ell(\cdot, s)$ used in the construction above themselves serve as inner barriers which last for a time period of length at least $-s$, by the avoidance principle we obtain the following lower bound on the time of extinction which serves as a quick plausibility check (these flows should be increasingly long lived). Likewise we obtain the upper bound by using a large cylinder of radius $\ell(s)$ as an outer barrier. 
    \begin{lem}\label{existence_time_estim}The time of existence of $\Sigma_s^{m^0}$, $\omega_{m^0}(s)$, satisfies $-s \leq \omega_{m^0}(s) \leq \ell(s)^2$. 
    \end{lem}     
   
 By Theorem \ref{aag_thm2}, the number of critical points of the ``weak'' profile curve cannot increase along the flow. In particular, the local minimum $m(t)$ of the graph of $\left(\Gamma^{m^0}_s\right)_t$ is well-defined, at least until the time this local minimum disappears. %Roughly speaking, this would mean that the flow of the stacked pancakes transitions into a single wider pancake. 
    By reflection symmetry for our purposes we may discuss $m(t)$ through this time by just defining it as the value of the (at least locally defined) profile function over the $x$-axis at $x = 0$, even if the flow is convex, where we take it to be zero if a neckpinch occurred or is occurring. We also denote the maximum height of $\left(\Gamma^{m^0}_s\right)_t$ by $M(t)$. Since these notions are used often, we highlight them in a definition: 
    \begin{defn}\label{minmax_defn} Where $\left(\Sigma_s^{m^0}\right)_t$ is a reflection and rotationally symmetric flow with locally graphical profile curve $\left(\Gamma^{m^0}_s\right)_t$ as described above, we write:
    \begin{enumerate}
    \item $m(t) = r(\{ x =0 \} \cap \left(\Gamma^{m^0}_s\right)_t)$, equal to zero if this set is empty. 
    \item $M(t) = \sup\limits_{x_0} r(\{ x =x_0 \} \cap \left(\Gamma^{m^0}_s\right)_t)$
    \end{enumerate}
    where $r$ is the height function over the axis of rotation. 
    \end{defn}

    \begin{figure}[h]\centering
    \includegraphics[width=1\textwidth]{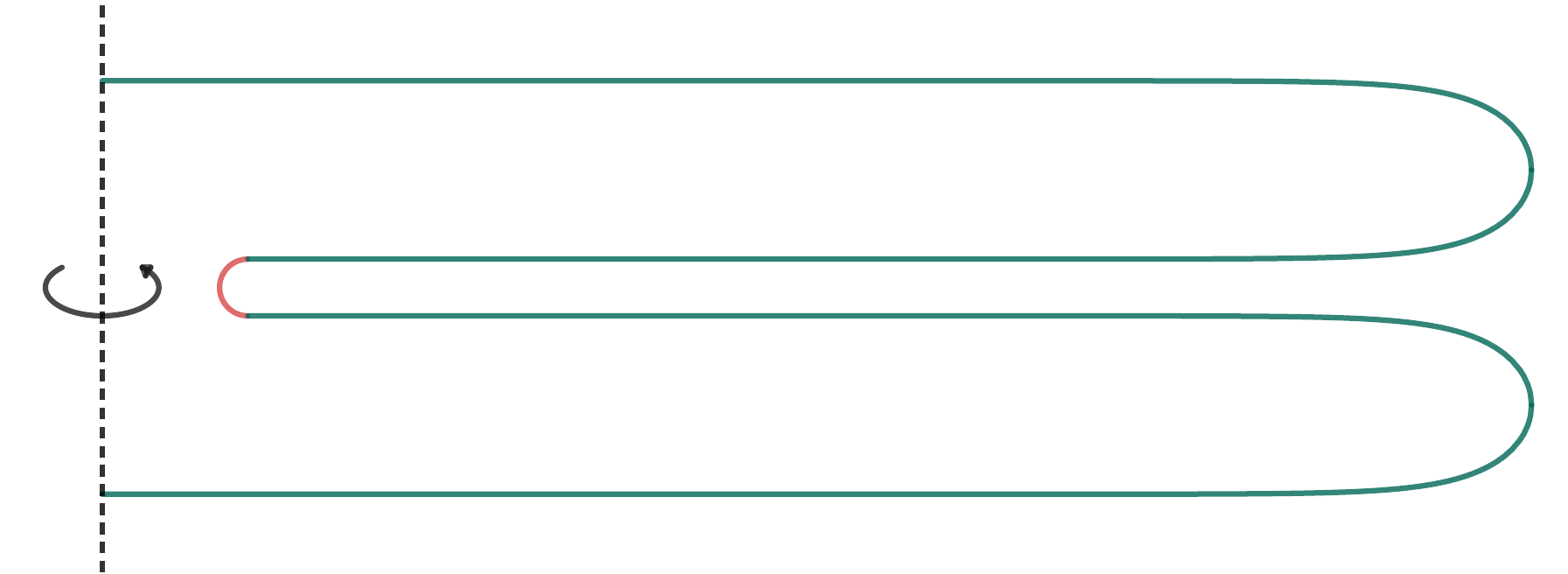}
    \caption{A schematic of the profile curve $\Gamma^{m^0}_s$; note that in actuality the profile curves are not exactly vertical for a given choice of $s$.}
    \label{profile_curve_fig}
    \end{figure}

Now we turn to actually obtaining a non-trivial ancient solution, worthy of the moniker ``stacked pancakes.'' This is by a continuity argument and towards that end we need a number of lemmas to show various objects/quantities are continuous in the initial neck parameter $m^0$, starting with the following: 

    \begin{lem}\label{continuous_dependence}
        For fixed $s < 0$, %let $\Sigma^m_s$, be the smooth $1$-parameter family of rotationally symmetric compact surfaces configured as two stacked pancakes described above. Then $(\Sigma^m_s)_t$, 
        the flows $\{(\Sigma^{m^0}_s)_t\}_{t\in[0,\omega_m(s))}$ are continuous in $m^0$ as Brakke flows, where $m^0 \in [\delta, \ell(s)]$. For any fixed time their supports are continuous in Hausdorff distance away from the rotation axis. 
    \end{lem}

    \begin{proof}
         As discussed above, Theorem $1.1$ in \cite{aag} tells us that for each $m^0$ the level set flows of $\Sigma_s^{m^0}$ are non-fattening and in fact by Theorem \ref{aag_thm3} the number of singular times is finite. In particular such flows are non--discrepant by Theorem 4.37 in \cite{LP} (see Section 4.3 in \cite{LP} for the definition of non--discrepancy). This implies that there is a unique unit regular, cyclic mod $2$ Brakke flow emanating from $\Sigma_s^{m^0}$ for fixed $s$ and $m^0$, and the support of this flow is exactly the non-fattening level set flow: in this case the outermost flows of $\Sigma_s^{m^0}$ correspond precisely to their level set flows, and as shown in Appendix B of \cite{HW}, these are the support of a unit regular and cyclic mod 2 Brakke flow. The uniqueness of this flow follows from Corollary G.5 in \cite{CCMS}.      
    
        It is easy to check that we have a uniform entropy bound at $t = 0$ by the design of the initial data and hence on  $\lambda\left(\left(\Sigma^{m^0_i}_{s}\right)_t\right)$ for $t > 0$ by monotonicity of the entropy. Thus, given any sequence $m^0_i \to m^0_\infty$, we may employ Brakke's compactness theorem to ensure that a subsequence of the Brakke flows $\left(\Sigma^{m^0_i}_{s}\right)_t$ converges to a Brakke flow of $\Sigma^{m^0_\infty}_s$. This limit Brakke flow will be unit regular and cyclic mod 2, so by uniqueness it must agree with $\left(\Sigma^{m^0_\infty}_s \right)_t$, thought of as a Brakke flow. The claim on continuity as Brakke flows then follows by a compactness contradiction argument, with the Hausdorff convergence following easily from Lemma \ref{Hau_conv}.
    \end{proof}

 In the sequel an important quantity for us will be the ``stopping time" in setting the time domain of the old-but-not-ancient flows. In the following, the dimensional constant $c_n > 1$ is a factor which we'll specify below, in Lemma \ref{silky_smooth}:
 \begin{defn} For a given $m^0 \in [\delta, \ell(s)]$ we write $T_{m^0} (= T_{m^0}(s))$ for the first time that $M(T_{m^0}) = c_n \cdot 10^{10}$. If no such time exists, we set $T_{m^0} = \infty$.
 \end{defn}
 Note that if $\ell(s)$ is sufficiently large such a time exists (and in fact is the only such time) using cylinderical barriers. We next argue that $T_{m^0}$ is continuous in $m_0$, where in the lemmas below $s$ will be fixed but large enough that $T_{m^0}$ is well defined, say $\ell(s) > 2 \cdot c_n10^{10}$:
 \begin{lem}\label{stop_cont} For fixed $|s| \gg 0$, consider the flows $(\Sigma^{m^0}_s)_t$ from above. Then $T_{m^0}$ is continuous in the initial neck parameter $m^0$ for $m^0 \in [\delta, \ell(s)]$.
\end{lem}
 
\begin{proof} By the rotational symmetry of $(\Sigma^{m^0}_s)_t$, it suffices to prove the result for $\left(\Gamma^{m^0}_s\right)_t$, the flow of the profile curves. Considering a sequence of initial neck parameters $m^0_i \to m^0_\infty$, we wish to show that $T_{m_i^0} \to T_{m^0_\infty}$. To do so we begin with a couple of observations. We recall that Lemma \ref{continuous_dependence} tells us $\left(\Gamma_s^{m^0_i}\right)_{T_{m^0_\infty}} \to  \left(\Gamma_s^{m^0_\infty}\right)_{T_{m^0_\infty}}$ as measures and in fact their supports must converge in the Hausdorff distance away from the axis of rotation. This implies that $M_i(T_{m_\infty^0})$ must converge to $M_\infty(T_{m_\infty^0}) = c_n10^{10}$ as $i \to \infty$, where $M_i(t)$ denotes the maximum distance from the axis of rotation attained on $\left(\Gamma_s^{m^0_i}\right)_t$. More generally, we see that $M_i(t) \to M_\infty(t)$ for arbitrary times $t$. Next, by the curvature bounds resulting from Theorem \ref{aag_grad} and Ecker--Huisken estimates (see also the surrounding discussion), we see that, for all $i$, $M_i(t)$ is a Lipschitz function in $t$ as long as the flow is not extinct. With this in mind, by considering large cylindrical barriers (depending on $s$), we have at (the a.e.) differentiable times there is a uniform positive constant $C_1$ for which $M'_i(t) < -C_1$ by Hamilton's trick (Lemma 2.1.3 in \cite{mant}).

 With these ingredients in hand, suppose for the sake of contradiction that $T_{m^0_i} \not\to T_{m^0_\infty}$. Extracting a converging subsequence (as by a barrier argument this sequence of times is certainly bounded), we may suppose that $T_{m^0_i} \to T' \neq T_{m^0_\infty}$. Since $M_i(T_{m_\infty^0}) \to M_\infty(T_{m_\infty^0}) = c_n10^{10}$ and $M_i(T_{m^0_i}) = c_n10^{10}$, we trivially have that 
 \begin{equation*}
    \lim\limits_{i \to \infty} |M_i(T_{m_i^0}) - M_i(T_{m_\infty^0})| \to 0,     
\end{equation*}
 which we claim implies $M_\infty(T') = M_\infty(T_{m_\infty^0})$. To justify this, first write
 \begin{equation} 
 M_i(T_{m_i^0}) - M_\infty(T') = (M_i(T_{m_i^0})  - M_i(T')) + (M_i(T')  - M_\infty(T')).
 \end{equation}
 Note that $M_i(T')  - M_\infty(T')$ tends to zero by Lemma \ref{continuous_dependence}, while the first term $ M_\infty(T_{m_i^0})  - M_\infty(T')$ tends to zero due to $T_{m_i^0} \to T'$ and the previously mentioned gradient bounds (recalling $M_i(T_{m_i^0}) = c_n10^{10}$ for all $i$). In particular, if this did not happen, then there would be some $0 <\epsilon \ll 1$ so that $|M_i(T_{m_i^0})  - M_i(T')| > \epsilon$ for all $i \gg 1$, after passing to a subsequence. Because the $M_i(t)$ are continuous over the interval $[T_{m^0_i}, T_{m^0_\infty}]$ (or $[T_{m^0_\infty}, T_{m^0_i}]$ if $T_{m^0_\infty} < T_{m^0_i}$) the images of $M_i(t)$ must contain a uniformly positive subinterval, since $\epsilon \ll c_n10^{10}$, of length at least $\epsilon$. Working in these intervals, by the curvature bounds resulting from Theorem \ref{aag_grad} and Ecker--Huisken estimates applied about the maxima, we can see there is some uniform constant $C$ so that $\epsilon < C|T_{m_i^0}  - T'|$ which, since $T_{m_i^0}  \to  T'$, gives a contradiction for $i \gg 1$. Finally, since we saw above that $M_i(T_{m_\infty^0}) \to M_\infty(T_{m_\infty^0})$, we have the claimed equality $M_\infty(T') = M_\infty(T_{m_\infty^0})$.

In particular, if $T' \neq T_{m_\infty^0}$ then $M_\infty(t)$, as it is non-increasing, is constant on some time interval, which can be dealt with by a barrier argument. Specifically, since $M_\infty(T') = M_\infty(T_{m_\infty^0}) = c_n10^{10}$ if $T' > T_{m^0_\infty}$, then the bound $M'_i(t) < -C_1$ from the cylindrical barrier implies that $M_i(T_{m_\infty^0}) < c_n10^{10} -\delta$ for some $\delta > 0$ for $i \gg 1$, which yields a contradiction. We similarly find a contradiction if $T' < T_{m^0_\infty}$. \end{proof}

Our final preparatory continuity lemma is the following: 

\begin{lem}\label{m(T)_cont} For fixed $|s| \gg 0$, consider the flows $(\Sigma^{m_0}_s)_t$ from above. Then $m(T_{m^0})$ is continuous as a function of $m^0$. 
\end{lem} 
\begin{proof} As before, we consider a sequence for which $m^0_i \to m^0_\infty$. As in the previous argument we note that along this family of flows $m(t)$ is not a fixed function, but depends implicitly on $m^0$. To make this clear in this proof, we will denote by $m_i(t)$ the local minimum from the axis of rotation on the flow $(\Gamma^{m^0_i}_s)_t$ and $m_\infty(t)$ on the limit $(\Gamma^{m^0_\infty}_s)_t$  (extended in the convex and disconnected cases as indicated in Definition \ref{minmax_defn}). We now need to show that 
\begin{equation*}
    |m_i(T_{m^0_i}) - m_\infty(T_{m^0_\infty})| \to 0.
\end{equation*}
For this, first note that, similarly to the proof of Lemma \ref{stop_cont}, we can write
\begin{equation*}
    m_i(T_{m^0_i}) - m_\infty(T_{m^0_\infty}) = (m_i(T_{m^0_i}) - m_i(T_{m^0_\infty})) + (m_i(T_{m^0_\infty}) - m_\infty(T_{m^0_\infty})).
\end{equation*}
By Lemma \ref{continuous_dependence} (one can check even if $m_\infty(T_{m^0_\infty)} = 0$) we have for the second term:
\begin{equation}\label{m(T)_cont_first_term}
    |m_i(T_{m^0_\infty}) - m_\infty(T_{m^0_\infty})| \rightarrow 0.
\end{equation}
as $i \rightarrow \infty$. For the other term, suppose for the sake of contradiction there is an $0 < \epsilon \ll 1$ (and a subsequence, which after relabeling we may suppose is the entire sequence) for which $|m_i(T_{m^0_i}) - m_i(T_{m^0_\infty})| > \epsilon$. By \eqref{m(T)_cont_first_term} and Lemma \ref{stop_cont}, we know that $m_i(T_{m^0_\infty}) \to m_\infty(T_{m^0_\infty})$ and $T_{m_i^0} \to T_{m^0_\infty}$. With this in mind we split our reasoning (mostly for exposition related to the gradient estimates) into two cases: $m_\infty(T_{m^0_\infty}) > 0$ and $m_\infty(T_{m^0_\infty}) = 0$. For each of these cases, we will arrive at a contradiction.

In the first case, we find a contradiction via the gradient bounds from Theorem \ref{aag_grad} (again, using that at this step $s$ is fixed and the remarks surrounding it). To elaborate, in the following take $i \gg 1$ so that $|m_i(T_{m^0_\infty}) -m_\infty(T_{m^0_\infty})| < \epsilon/2$. By shrinking $\epsilon$ if necessary, this also implies $m_\infty(T_{m^0_\infty}) < 2m_i(T_{m^0_\infty})$ so that the values $m_i(T_{m^0_\infty})$ are uniformly positive in $i \gg 1$, in fact so that say $m_i(T_{m^0_\infty}) > \epsilon$. Its easy to see that the $m_i(t)$ are continuous, so then for $i$ large over the interval $[T_{m^0_i}, T_{m^0_\infty}]$ (or $[T_{m^0_\infty}, T_{m^0_i}]$ if $T_{m^0_\infty} < T_{m^0_i}$) the images of $m_i(t)$ must contain uniformly positive subintervals of length at least $\epsilon/2$. The gradient bounds from Theorem \ref{aag_grad} and Ecker--Huisken estimates applied about the points where $m_i(t)$ are attained for these intervals imply that there is some constant $C$, independent of $i$, for which $\epsilon/2 < C|T_{m^0_\infty} - T_{m^0_i}|$. Since $|T_{m^0_\infty} - T_{m^0_i}| \to 0$, this yields a contradiction for $i \gg 1$. 

In the latter case ($m_\infty(T_{m^0_\infty}) = 0$) we can still proceed basically as above, but we split this case into its own because our curvature estimates degenerating about the axis of rotation may give pause for concern. We recall that for $i \gg 1$, $|m_i(T_{m^0_\infty})| < \epsilon/2$. Thus, since $|m_i(T_{m^0_i})| > \epsilon$, for $i \gg 1$ over the interval $[T_{m^0_i}, T_{m^0_\infty}]$ (or $[T_{m^0_\infty}, T_{m^0_i}]$ if $T_{m^0_\infty} < T_{m^0_i}$) $m_i(t)$ must cover an interval of length at least $\epsilon/2$; in particular the interval $(\epsilon/2, \epsilon)$. At least for this interval, when $\epsilon/2 < m_i(t)  < \epsilon$, we have uniform gradient estimates over $x = 0$ so as before there is some constant $C$ independent of $i$ for which $\epsilon/2 < C|T_{m^0_\infty} - T_{m^0_i}|$. Since $|T_{m^0_\infty} - T_{m^0_i}| \to 0$ this yields a contradiction for $i \gg 1$ as before.\end{proof}
 
We now use the lemmas above to construct our old--but--not--ancient flows via a continuity argument; below $d_n > 1$ is another dimensional constant that we take at least large enough that for $s \leq -d_n 10^{100}$ we have $\ell(s) > 2 \cdot c_n 10^{10}$:
     \begin{prop}\label{good_OBNA_flow_existence} For each $s \leq -d_n10^{100}$, there is a choice of initial neck parameter $\overline{m}^0 = \overline{m}^0(s) \in [\delta, \ell(s)]$ so that the flow $\left(\Sigma^{\overline{m}^0}_{s}\right)_t$, $t \in [0,T_{\overline{m}^0}]$, is smooth, rotationally and reflection symmetric, trapped in a slab of width less than $10$ and  at $t = T_{\overline{m}^0}$ is connected with $M(T_{\overline{m}^0}) = c_n 10^{10}$ and $1/2 < m(T_{\overline{m}^0}) < 1$. 
    \end{prop}
    \begin{proof}
    By a barrier argument using the Angenent torus and the design of $f_{m^0}$, when $m^0 \ll 1$ (shrinking $\delta$ as needed) we know a neckpinch will quickly form, while if $m^0 \gg 1$ the flow will quickly become convex using that as $m^0 \to \ell(s)$ the limit is convex (in particular, one may employ Theorem A.2 in \cite{MP1}). For $s\le -d_n10^{100}$, let $T_{m^0}(s)$ be the first time that $M(T_{m^0}(s))=c_n10^{10}$ as before. We see that for $m^0 \ll 1$ and for $m^0 \gg 1$ the flow will have either one or two components by time $T_{m^0}(s)$ and appear roughly as in Figure \ref{neck_dichotomy_fig} below. With this in mind, we define the following sets:
        \begin{align*}
        N_\alpha &:= \left\{m^0\in [\delta,\ell(s)]: (\Sigma_s^{m^0})_{T_{m^0}(s)}\mathrm{~has~two~connected~components}\right\},\\
            N_\beta &:= \left\{m^0\in [\delta,\ell(s)]: (\Sigma_s^{m^0})_{T_{m^0}(s)}\mathrm{~is~connected~and~convex}\right\}.
        \end{align*}        
            
        \begin{figure}[h]
          \centering
          \begin{tikzpicture}[rotate=-90]
            %Plots the profile curve for eta small
            \draw[very thick] (1.5,6) arc (0:180:1.5 and 3);
            \draw[very thick] (-3,6) -- (3,6)  node [midway] {\AxisRotator[rotate=-90]};  

            %Plots the profile curve for eta large
            \draw[very thick] (-1,0) arc (0:180:0.5 and 1.5);
            \draw[very thick] (2,0) arc (0:180:0.5 and 1.5);
            \draw[very thick] (-3,0) -- (3,0)  node [midway] {\AxisRotator[rotate=-90]};  
          \end{tikzpicture}
          \caption{Qualitative behavior of $(\Sigma^{m^0}_s)_{T_{\overline{m}^0}(s)}$ for $m^0 \ll 1$ on the left and for $m^0 \gg 1$ on the right.}
          \label{neck_dichotomy_fig}
        \end{figure}

        From the discussion above, $N_\alpha$ and $N_\beta$ are both nonempty and correspond to $m(T_{m^0}) = 0$ and $m(T_{m^0}) = c_n10^{10}$ respectively (recalling above how we defined $m$ when there was no local minimum). Thus Lemma \ref{m(T)_cont} and the intermediate value theorem tell us there must be some $\overline{m}^0 \in [\delta, \ell(s)]$ for which         
        \begin{equation}\label{min_bound}
            \frac{1}{2} < m(T_{\overline{m}^0}) < 1.
        \end{equation}
        Note that by Theorem \ref{aag_thm3} the flow $(\Sigma^{\overline{m}^0}_s)_{t}$ up to time $T_{\overline{m}^0}$ must be smooth. The other properties are similarly easy to observe.
        \end{proof}
As a consequence of the above we produce smooth flows $\left(\Sigma^{\overline{m}^0}_{s}\right)_t$ which at their end time, $T_{\overline{m}^0}(s)$, are uniformly bounded and uniformly nonconvex. Of course, all such flows have graphical profile curves. Finally, we can take a limit of these flows to obtain an ancient Brakke flow:
\begin{cor}\label{ancient_flow_existence}
       There exists an ancient unit regular, cyclic mod 2 Brakke flow $B_t$, $t \in (-\infty, 0]$ that is contained in a slab and given by a limit of smooth rotationally and reflection symmetric flows with graphical profile functions such that $\supp(B_0)$ is non-empty, compact, connected, and is not the boundary of a convex domain. 
    \end{cor}

\begin{proof}
        Consider the smooth flows $(\Sigma^{\overline{m}^0(i)}_{s_i})_{t}$ with $s_i := -d_n10^{i+100}$ and corresponding neck parameters $\overline{m}^0(i) := \overline{m}^0(s_i)$ from Proposition \ref{good_OBNA_flow_existence}, where $t \in [0, T_{\overline{m}^0(i)}]$. For each flow, recenter time so that $T_{\overline{m}^0(i)} = 0$ to obtain a sequence of flows $M^i_t$ and note that the $M^i_0$ are each uniformly bounded (in particular $M_i(0) = c_n10^{10}$), connected (by \eqref{min_bound}), and do not bound a convex domain (in fact, they are a uniform distance away from any such set in the Hausdorff distance). By the choice of $T_i := T_{\overline{m}^0(i)}$ and that $s_i \rightarrow -\infty$ as $i \rightarrow \infty$, the flows $M^i_t$ are defined on time intervals $[-T_i, 0]$ with $T_i \to \infty$ using inner pancake barriers similar to the proof of Lemma \ref{existence_time_estim}. One can check by the design of the initial data and monotonicity that we have the entropy $\lambda(M^i_t)$ is uniformly bounded, so by the Brakke compactness theorem we can extract a subsequence of the Brakke flows $M^i_t$ which converges to a limiting ancient Brakke flow $B_t$. Since each of these are unit regular and cyclic the limit will be too. Note by lemma \ref{Hau_conv}, we also have the convergence is in Hausdorff topology in the manner indicated therein from which it easily follows $B_t$ has the properties claimed above from the design of the $(\Sigma^{\overline{m}^0(i)}_{s_i})_{t}$.
    \end{proof}

    \begin{rem}
        We will frequently use that the flows $M^i_t$ approximate the Brakke flow $B_t$ to study this limit. As mentioned in the proof above, this implies that their supports converge in the Hausdorff distance on bounded domains $U$ as in Lemma \ref{Hau_conv}. However, the rate of convergence depends on $U$ and $t$ and in particular a priori there is not necessarily an index $k$ large enough for which $M^k_t$ is globally close to $B_t$.

        %More precisely, in the compactness argument in the proof of Corollary \ref{ancient_flow_existence}, one really takes a subsequence of the flows $M^i_t$ which converges to a limit flow $B_{\ell,t}$. 

        %This limit $B_{\ell,t}$ is defined on $B(0,\ell) \times [-\ell, 0]$ because, for $i \gg 1$, the time domain on which $M^i_t$ is defined contains the interval $[-\ell, 0]$. By a diagonalization argument one can then obtain an ancient flow $B_t$, which is approximated by the old-but-not-ancient solutions $M^i_t$ in the sense that for any fixed time interval there is a fixed (sub)sequence which converges to $B_t$ on any  parabolic cylinder $B\left(0,\ell\right) \times [-\ell, 0]$, the rate of convergence (potentially) depending on $\ell$. 
    \end{rem} 
    \begin{rem} For the sake of simplicity we made a number of concrete choices above, but it is easy to see we may freely vary the pancake thickness and separation, as well as the end time neck width, when specifying the initial data for $(\Sigma^{\overline{m}^0(i)}_{s_i})_{t}$. 
    \end{rem}
\subsection{Regularity and other properties of the ancient Brakke flow $B_t$.} Roughly speaking, the weak ancient flow from Corollary \ref{ancient_flow_existence} has the properties we wish, but our final goal is to upgrade it into a ``true'' smooth ancient solution to MCF. To set some notation, we will often write:
       \begin{itemize}
    \item $\Gamma^i_{t}$ for the profile curves of the old-but-not-ancient solutions $M^i_{t}$ of Corollary \ref{ancient_flow_existence}, which are time shifts of $\left(\Gamma^{\overline{m}^0(i)}_{s_i}\right)_{t}$, $\left(\Sigma^{\overline{m}^0(i)}_{s_i}\right)_{t}$ respectively by $-T_{\overline{m}^0(i)}$.
    
    \item Analogously, $\supp(\Gamma^\infty_t)$ will denote what one rotates around the axis of rotation to get $\supp(B_t)$.
    \end{itemize}
       
       With this in mind, we first prove the following:

    \begin{lem}\label{silky_smooth}
        For all times $t\le 0$, $\supp(B_t)$ is compact and connected.
    \end{lem}

    \begin{proof} We break the proof up into a number of steps, where we first discuss connectedness. Throughout, we will use $t^\ast$ to denote some fixed negative time.

        \textbf{Claim 1: If $\supp(B_{t^*})$ is disconnected it must be non-compact.}
        
         Note that, by Corollary \ref{ancient_flow_existence}, $\supp(B_0)$ is connected. Thus, if $\supp(B_{t^*})$ is disconnected then the number of connected components must decrease at a later time. If $\supp(B_{t^*})$ is compact, then the avoidance principle, which holds for Brakke flows, tells us by using a plane placed between the connected components as a barrier that $\supp(B_{t^*})$ cannot become connected at later times. This yields a contradiction. \hfill \qed %\textbf{(end of claim)}

        In light of Claim $1$, we may suppose that $\supp(B_{t^*})$ is non-compact in addition to being disconnected.

        \textbf{Claim 2: If $\supp(B_{t^*})$ is non-compact and disconnected, it must have ``cusps'' in the sense that there are connected components which approach each other at spatial infinity.}%Moreover, these cusps must be ancient in that they exist for all $t \in \left(-\infty, t^\ast\right)$.}

        We recall the avoidance principle for non-compact flows (for instance, see Theorem 1 in \cite{White_avoidance}), which gives that if two noncompact level set flows in $\R^{n+1}$ (which Brakke flows are) are a uniform distance apart then they stay disjoint, and recall that the  approximating flows $\Gamma^i_{t^\ast}$ are all graphical over the axis of rotation and trapped in a slab of uniform width. We then see that the only way the number of connected components of $\supp\left(\Gamma^\infty_{t^\ast}\right)$ could decrease is if two (non-compact) components of $\supp(\Gamma^\infty_{t^*})$ are asymptotic to each other at spatial infinity along a line $\ell$ perpendicular to the axis of rotation to form what we'll call a cusp. Indeed in this case, mass may rush in from spatial infinity through such a cusp to form a cap joining the two components sometime after, similar to Example 7.3 in \cite{I2}. In fact, by using large spheres as barriers, cusps cannot form at finite times along the flow by the avoidance principle. In other words, such cusps must be ancient in the sense that they correspond to cusps which exist for all times $t \in (-\infty, {t^*})$. \hfill \qed %\textbf{(end of claim)} 

        Next we argue that such ancient cusps cannot exist. As an aside, one approach is to first argue that such cusps bound regions of finite area. This uses that the flow must eventually be compact by the design of the ancient Brakke flow $B_t$ at $t = 0$ and that the rate of change of area bounded by a curve evolving under either CSF or the forced flow \eqref{rotsym_evol} is well controlled by the turning number. Then, using the ``slingshots'' of Bourni and Reiris from \cite{BR}, one can see that the approximating flows $\Gamma^i_t$ must in fact lie in a uniformly bounded region shortly after a given time $\overline{t} < t^*$. This contradicts that the cusp must be ancient (as observed in the argument above) if $\overline{t}$ is taken to be sufficiently negative. In the following claims, we detail a different approach using the Sturmian principle. In the ancillary claim below, we only consider $i$ large enough that $M^i_t$ to be defined at $t = t^*$.

        \textbf{Claim 3: For $i$ sufficiently large, $m_i(t^*)$ is bounded above by $c_n10^{10}$.}

        In the following we recall, say from section 4 of \cite{aag}, that the thickness of a catenoid is a linearly increasing function of its neck width. Now, suppose on the contrary that $m_i(t^*) > c_n10^{10}$ for some subsequence, which we then relabel to the original, of $\Gamma_{t^\ast}^i$. From this and since the $\Gamma^i_t$ are all trapped in slabs of width less than 10 (to give a crude but sufficient bound) if we adjust $c_n$ to be sufficiently large, which we'll potentially take larger yet in the next claim, there exists a catenoid of thickness greater than 10 whose profile curve $W$ intersects $\Gamma^i_{t^\ast}$ at exactly two points less than distance $c_n10^{10}$ from the axis of rotation and whose neck width $m_W$ is strictly larger than 1. In fact we may arrange that for all $t\in [t^*, 0]$ that $W$ could only possibly intersect $\Gamma^i_{t}$ at points less than distance $c_n10^{10}$ from the axis of rotation. With this in mind since $M_i(t) \geq c_n10^{10}$ for $t \leq 0$ by reflection symmetry and the Sturmian principle $m_i(t) > m_W > 1$ for all $t \in [t^*, 0]$; otherwise the minimum of $\Gamma^i_t$ would ``poke through'' $W$ and cause the number of intersection points to increase. This contradicts the design of $M^i_0$, however.  \qed\\

        As mentioned in the introduction, this seems to imply that the backwards limit, if smooth, should contain a catenoid. With this claim in hand, we are prepared to rule out cusps:

        \textbf{Claim 4: The cusps from Claim 2 do not exist.}
     
         We will suppose that such a cusp exists and obtain a contradiction. Below, it will often be helpful to consider the unshifted flows $\left(\Gamma^{\overline{m}^0(i)}_{s_i}\right)_{t}$, $\left(\Sigma^{\overline{m}^0(i)}_{s_i}\right)_{t}$ where we recall $\Gamma^i_t = \left(\Gamma^{\overline{m}^0(i)}_{s_i}\right)_{t + T_{\overline{m}^0(i)}}$ and similarly for $M^i_t =\left(\Sigma^{\overline{m}^0(i)}_{s_i}\right)_{t + T_{\overline{m}^0(i)}}$ because we will arrange ``comparison pancakes'' relative to their initial data. 
         
         Now, recall that by the graphicality of $\Gamma^i_t$ we see that a cusp as described in Claim 2 is asymptotic to a line $\ell$ which is perpendicular to the axis of rotation. With this in mind, below we consider comparison pancakes $P^i$ (with profile curve $\gamma^i$) of asymptotic thickness $\pi$, rotationally symmetric about the axis $r = 0$ such that:
        \begin{itemize}
        \item The maximum distance $d^i$ of $\gamma^i$ to the axis of rotation satisfies $d^i > M_i(0)$. 
        \item The ``tip'' of $\gamma^i$, i.e. the point where $d^i$ is achieved, is a distance less than $1/100$ from $\ell$. 
        \item $\gamma^i$ meets $\left(\Gamma^{\overline{m}^0(i)}_{s_i}\right)_{0}$ at up to two points; this holds generically by the asymptotics of the pancakes and design of initial data.

        \end{itemize}
        Denote by $P^i_t, \gamma^i_t$ the flow of $P^i, \gamma^i$ respectively  and $d^i(t)$ the flow of $d^i$. Our first goal is to get uniform bounds on $d^i(t)$. Below we will write $T_i := T_{\overline{m}^0(i)}$ for short.
        
        To start, we claim that $T_i$ is at least as large as the time it takes for $d^i(t)$ to equal $c_n10^{10} + 2(d^i(0) - M_i(0))$. For this, we will use the following crude fact on the asymptotics of the standard ancient pancake $P_s$, $s \in (-\infty, 0)$ of asymptotic thickness $\pi$ (as in section \ref{old_not_ancient_construction}). For $s \ll 0$ (or in other words $\ell(s)$ sufficiently large), the inwards speed of the tip of $P_s$, $-\frac{d \ell}{dt}$, is crudely bounded between $1$ and $2$ from the asymptotics in \cite{blt_pancakes}; roughly speaking, for $|s|$ large the profile curve of a standard pancake is close to that of a grim reaper of width $\pi$ which translates under the CSF with speed 1, and the extra terms in equation \ref{rotsym_evol} from rotating the profile curve decay $\sim (n-1)/\ell(s)$ away from the axis of rotation. 
        
        In the following, we take $c_n, d_n$ to ensure we may use the speed bounds above. We can always take these constants larger as needed, and just depend on the $n$--dimensional standard pancake; doing so will not affect the argument of the previous claim, either. First, since (two) pancakes $P_{s_i}$, $s_i \leq -d_n10^{100}$, were used in constructing $(\Sigma^{\overline{m}^0(i)}_{s_i})_{t}$ the age of the comparison pancake $P^i$ is no older than $s_i - (d^i(0) - M_i(0))$ from the speed lower bound, by which we mean up to translation $P^i$ above is contained in $P_{s_i - (d^i(0) - M_i(0))}$. Now, using shifts of $P_{s_i}$ along the axis of rotation as inner barriers as indicated for Lemma \ref{existence_time_estim} we see $T_i$ is at least as large as the time $\overline{T_i}$ it takes for the tip distance of $P_{s_i + t}$ to equal $c_n10^{10}$. The distance between the tips of $P_{s_i + \overline{T_i}}$ and $P_{s_i - (d^i(0) - M_i(0)) + \overline{T_i}}$ is at most $2(d^i(0) - M_i(0))$ by the crude speed bound above, giving by the convexity of the pancake solution that $T_i$ is at least as large as the time it takes for $d^i(t)$ to equal $c_n10^{10} + 2(d^i(0) - M_i(0))$.
        
        As a simple observation illustrating this, if we initially arrange for $d^i(0) -M_i(0)< 1$ then the tip of $\gamma^i_{T_i}$ will be a uniformly bounded distance from the axis of rotation, we emphasize independent of $i$. Similarly, for some time $t^\ast < 0$ we can crudely estimate:
        \begin{equation*}
            d_i(t^\ast + T_i) < c_n10^{10} + 2(|t^\ast| + (d^i(0) - M_i(0))).
        \end{equation*}
        Where $i$ is large enough that $t^\ast + T_i > 0$. In particular, we can arrange $\gamma^i$ so that $d^i(t^\ast + T_i) < c_n10^{10} + 2|t^\ast| + 2$, no matter how large $i$ is (in fact, the flow $P^i_t$ could conceivably go extinct by this time, but for our purposes this is fine).
        
        With this in mind, fix a point $q \in \ell$, $r(q) \gg c_n10^{10} + 2|t^\ast| + 2$ so that, for $i \gg 1$, $\Gamma^i_{t^\ast} \cap B(q, 1) = \left(\Gamma^{\overline{m}^0(i)}_{s_i}\right)_{t^\ast + T_i} \cap B(q, 1)$ consists of at least two approximately (in $C^0$ distance) vertical line segments both of distance less than $1/100$ from $q$. Note this is possible due to Lemma \ref{Hau_conv} (and the assumed existence of the cusp). Choosing an index $k$ large enough that the above and the previous claim holds, to reiterate we have that if $d^k = d^{k}(0) < M_k(0) + 1$ then $d^k(t^\ast + T_{k}) <  c_n10^{10} + 2|t^\ast| + 2$. In particular, it is clear starting from this choice that we may simply increase $d^k$ appropriately so that the tip of $\gamma_{t^\ast + T_i}^{k}$ intersects $B(q,1)$ near $q$ -- we may also slightly shift $P^k$ along the rotation axis to not affect the number of intersection points with $\left(\Gamma^{\overline{m}^0(k)}_{s_k}\right)_{0}$. As a result of this, we see the tip intersects $\left(\Gamma^{\overline{m}^0(k)}_{s_k}\right)_{t^\ast + T_i}$ in $B(q, 1)$ at least twice; a possible configuration is illustrated in Figure \ref{Sturm_cusp}. By the reflection symmetry and connectedness of $\left(\Gamma^{\overline{m}^0(k)}_{s_k}\right)_{t}$ if these were the only two points of intersection with the pancake then we see $m_i(t^* + T_i) > r(q)$ contradicting the previous claim, giving that the number of intersection points of the profile curve of $\gamma^{k}_t$ with  $\left(\Gamma^{\overline{m}^0(k)}_{s_k}\right)_{t}$ must have then increased (there must be at least 3 points). This contradicts the Sturmian principle, and therefore the claim holds. \hfill \qed
            \begin{figure}[h]\centering
   \includegraphics[width=1\textwidth]{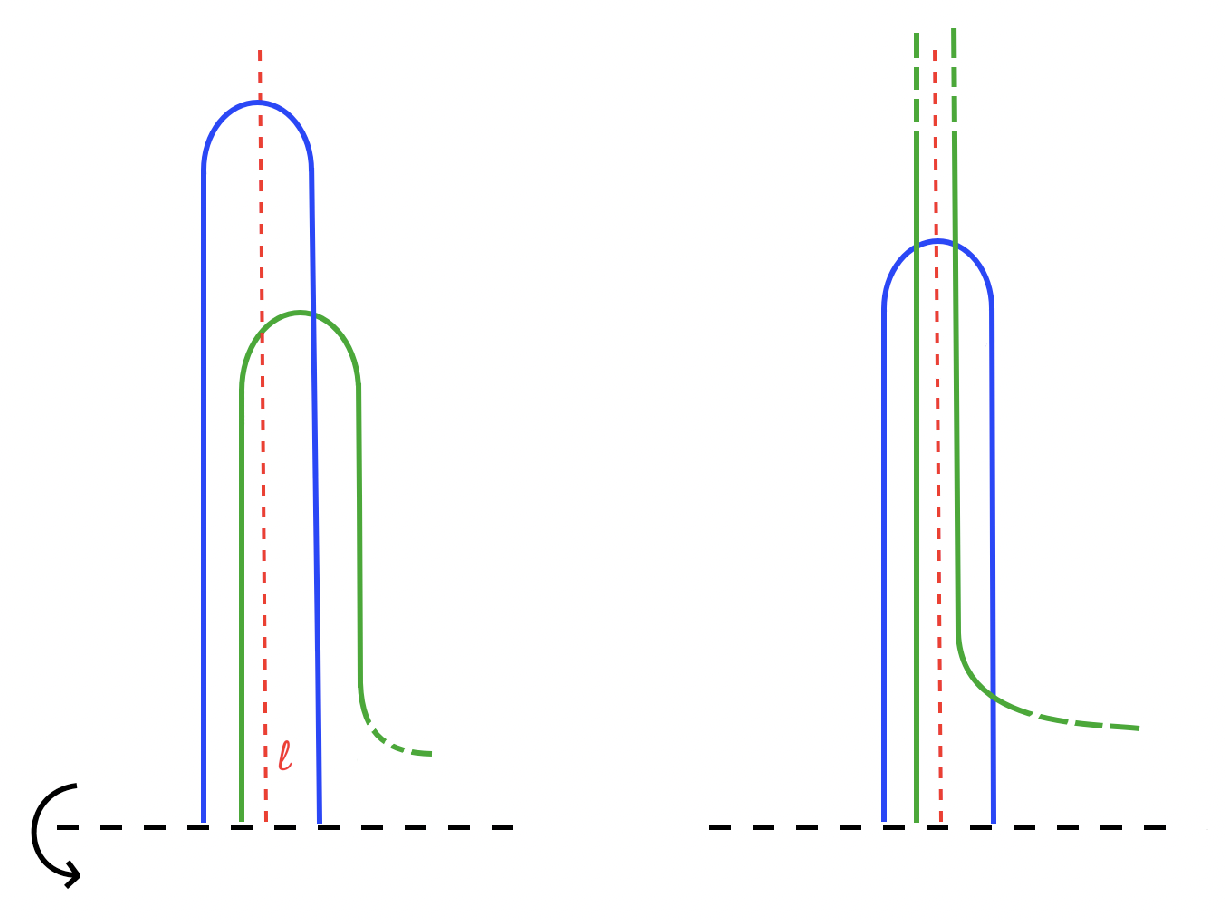}
    \caption{The profile curves in the diagram above show a potential configuration of the ``comparison pancake'' $\gamma_t^k$ in blue and  $\left(\Gamma^{\overline{m}^0(k)}_{s_k}\right)_{t}$ in green at times $t = 0$ and $t = t^\ast + T_k$ respectively (for illustrative purposes the relative scales differ, and only a piece of the curves are drawn). Here the number of intersection points would increase from $1$ to $3$, violating the Sturmian principle.}
    \label{Sturm_cusp}
\end{figure}
        
        It now remains to rule out $\supp(B_{t^*})$ being connected yet non-compact; this is similar to the previous steps: 

        \textbf{Claim 5: If $\supp(B_{t^*})$ is connected, it cannot be non-compact.} \\
Suppose for the sake of contradiction that $\supp(B_{t^*})$ and hence $\supp(\Gamma^\infty_{t^*})$ is connected and non-compact. First, using the unboundedness, by the graphicality of the profile curves $\Gamma_t^i$ and that they are trapped in a slab of uniform width we see that $\supp(\Gamma^\infty_{t^*})$ is asymptotic to a number of lines perpendicular to the axis of rotation; since $B_0$ is bounded by a proof by contradiction using the avoidance principle with appropriately thick inner pancake barriers the corresponding convergence of the $\Gamma^i_{t^*}$ to these lines must have multiplicity at least $2$. Note because each of the $\Gamma^i_{t^*}$ has exactly two maxima there are at most two such lines (and possibly just one, if the flow degenerates). In fact, since $\supp(\Gamma^\infty_{t^*})$ is connected and, again, that each of the $\Gamma_t^i$ are graphical one can see that $\supp(\Gamma^\infty_{t^*})$ must contain subrays, or ``prongs,'' of these asymptotic lines pointing out from the axis of rotation. By reflection symmetry and the bound on number of maxima the multiplicity of such prongs are either 2 or 4. These can then be ruled out reasoning similar to the previous claim.  \qed\\

To recap, we showed that if supp$(B_t)$ at a time $t^*$ is disconnected it must be non-compact and, using the design of the flow at $t = 0$, must be asymptotic to ``cusps''. We then ruled out such cusps via the Sturmian principle to see that the flow must be connected, by an argument which also rules out high multiplicity rays. From this it is easy to see that supp$(B_t)$ is bounded, giving the lemma. \end{proof}
Having established the boundedness and connectedness of the flow for any given time, we next establish regularity in part using Theorem \ref{aag_grad}. Of course, by design all the approximating flows are smooth but this doesn't seem to immediately rule out singularities in the limit. Some care is also needed to rule out singularities along the axis of rotation, by an argument which seems to be specific to dimensions $n \geq 3$; this is discussed more in remark \ref{dimrem} below.

\begin{lem}\label{smooth_seq} 
    The ancient Brakke flow $B_{t}$ is smoothly embedded for all $t \in (-\infty, 0]$.
    \end{lem}
    \begin{proof} 
    For concreteness, consider a time $t^* < 0$ and suppose it is a singular time for the limiting flow $B_t$. Denote by $p$ a point in the singular set of $\supp(\Gamma^\infty_{t^*})$, where we recall $\supp(\Gamma^\infty_t)$ is what one rotates around the axis of rotation to get $\supp(B_t)$. 
    
    First, we rule out singularities off the axis of rotation. Considering a time $\overline{t}$ shortly before $t^\ast$, Lemmas \ref{silky_smooth} and \ref{Hau_conv} applied at time $\overline{t}$ gives the smooth approximating flows $\Gamma^i_{\overline{t}}$ are all uniformly bounded for $i$ sufficiently large. This allows us to then apply Theorem \ref{aag_grad} along with the Ecker--Huisken estimates \cite{EH} as above to $\Gamma^i_{t^*}$ to see that about any point $p$ not along the axis of rotation $M^i_{t^*}$ has uniformly bounded curvature (and derivatives) for $i \gg 1$. This easily implies $B_{t^\ast}$ must in fact be smooth at any point away from the origin, or in other words there can be no off-axis singularities.

    We next claim that $p$ does not lie on the axis of rotation. Supposing that $p$ lies on the axis of rotation, we first rule out it being on the top or bottom caps of $\supp(B_{t^\ast})$ (defined in the obvious sense), inspired by Lemma 5.4 in \cite{aag}. To see this, since $M_i(t) > c_n 10^{10}$ for all $i$ and all $t < 0$ the flows of $M^i_t$ about their top and bottom caps can be written intrinsically locally as graphs in balls of uniformly large radius. Because the flows are trapped in slabs of uniform width, giving $C^0$ bounds in this context, the Evans-Spruck estimates (specifically, see Lemma 2.2 in \cite{aag}) thus give us uniform gradient, and hence higher order, bounds about their top and bottom caps, which implies that any tangent flow about the top and bottom caps of $B_{t^\ast}$ must contain a plane possibly with multiplicity. Taking a point $p$ on a cap, if the multiplicity is $1$ then $p$ is regular, so suppose the multiplicity is at least $2$. In this case though by graphicality of the $\Gamma_t^i$ there would then be a singular point away from the axis of rotation which we just ruled out. 

Thus $p$ must be a limit of the local minima of $\Gamma^i_{t^\ast}$, arguing by contradiction with Lemma \ref{Hau_conv}. In the following, recall that the $\Gamma^i_t$ are simply time shifts of the profile curves $\left(\Gamma^{\overline{m}^0(i)}_{s_i}\right)_t$ by $-T_{\overline{m}^0(i)}$ -- for the sake of exposition here it will be more clear to refer to these explicitly and their associated functions $m_i(t)$, $t \in [0, T_{\overline{m}^0(i)}]$. Also recall the necks of the initial data of these profile curves have uniformly controlled geometry in that the the partial profile functions $f_{\overline{m}^0(i)}$ from Section \ref{old_not_ancient_construction} satisfy:

    \begin{enumerate}[label=\roman*)]
        \item they are defined over a uniform interval on the axis of rotation,
        \item they obey uniform $C^1$-bounds,
        \item $\overline{m}^0(i) > \delta$ (suitable $\delta$ can be fixed independent of $i$ by the above)
    \end{enumerate}

    In light of this, and that catenoids lie in slabs in dimension $n \geq 3$, we can fix a catenoid C with profile curve W so that, for $i \gg 1$, 
    \begin{equation*}
        \left|W \cap \left(\Gamma^{\overline{m}^0(i)}_{s_i}\right)_{0}\right| = 2.
    \end{equation*}
 In particular, we can ensure that $m_W < \frac{\delta}{2}$, where $m_W$ is the neck radius of $W$.

    With the above in mind, since the singular point $p$ lying on the axis of rotation is a limit of the local minima of $\Gamma^i_{t^*}$ we see there is a sequence of times $t_i \rightarrow t^\ast$ such that $m_i(t_i) \rightarrow 0$. Writing for shorthand $T_i := T_{\overline{m}^0(i)}$ as before, therefore by the Sturmian principle and time recentering we can see that, for $i \gg 1$,   
    \begin{equation}\label{contra_setup}
        W \cap \left(\Gamma^{\overline{m}^0(i)}_{s_i}\right)_{t_i + T_i} = \emptyset.
    \end{equation}
        Indeed because each $\left(\Gamma^{\overline{m}^0\left(i\right)}_{s_i}\right)_t$ is compact and embedded, with boundary on the axis of rotation, and the profile curve $W$ is boundaryless we see their mod 2 intersection number is invariant under the flow, so it is trivial since it is at $t = 0$. So, since initially $\left|W \cap \left(\Gamma^{\overline{m}^0(i)}_{s_i}\right)_0\right| = 2$ by the Sturmian principle they can only intersect $W$ at either $0$ or $2$ points (counting with multiplicity) at a given time $t$, and this number is nonincreasing in time. With this in hand, for times $t$ at which $m_i(t) < m_W$ it is easy to see we must have that $\left(\Gamma^{\overline{m}^0(i)}_{s_i}\right)_t$ and $W$ are in fact disjoint, with the profile curves in the configuration indicated in Figure \ref{noaxis_sing}: as long as the number of intersection points is $2$, $\left(\Gamma^{\overline{m}^0(i)}_{s_i}\right)_t$ intersects the closed ``outer'' domain bounded by $W$ (the one homeomorphic to $((0,1) \times \R)^c \subset \R^2$) in a possibly degenerate embedded interval $I_t$, and by reflection symmetry the point in $\left(\Gamma^{\overline{m}^0(i)}_{s_i}\right)_t$ over $x = 0$ lays in $I_t$ so for such $t$ $m_i(t) \geq m_W$. However, since $m_i(T_i) > \frac{1}{2} > \delta$ by \eqref{min_bound} (shrinking $\delta$ if needed), we see that for $i \gg 1$ there must be some $\overline{t_i} \in (t_i, 0)$ at which $m_i(\overline{t_i} + T_i) = m_W$, i.e.:
    \begin{equation*}
        W \cap \left(\Gamma^{\overline{m}^0(i)}_{s_i}\right)_{\overline{t_i} + T_i} \neq \emptyset.
    \end{equation*}
    This contradicts the avoidance principle due to \eqref{contra_setup} and $t_i < \overline{t_i}$, and hence there can be no singularities along the axis of rotation.  As the $\Gamma^i_{t^*}$ are all graphs of locally uniformly bounded gradient away from the axis of rotation the limit profile curve is embedded, giving the lemma. \end{proof}
    
      \begin{figure}[h]\centering
    \includegraphics[width=1\textwidth]{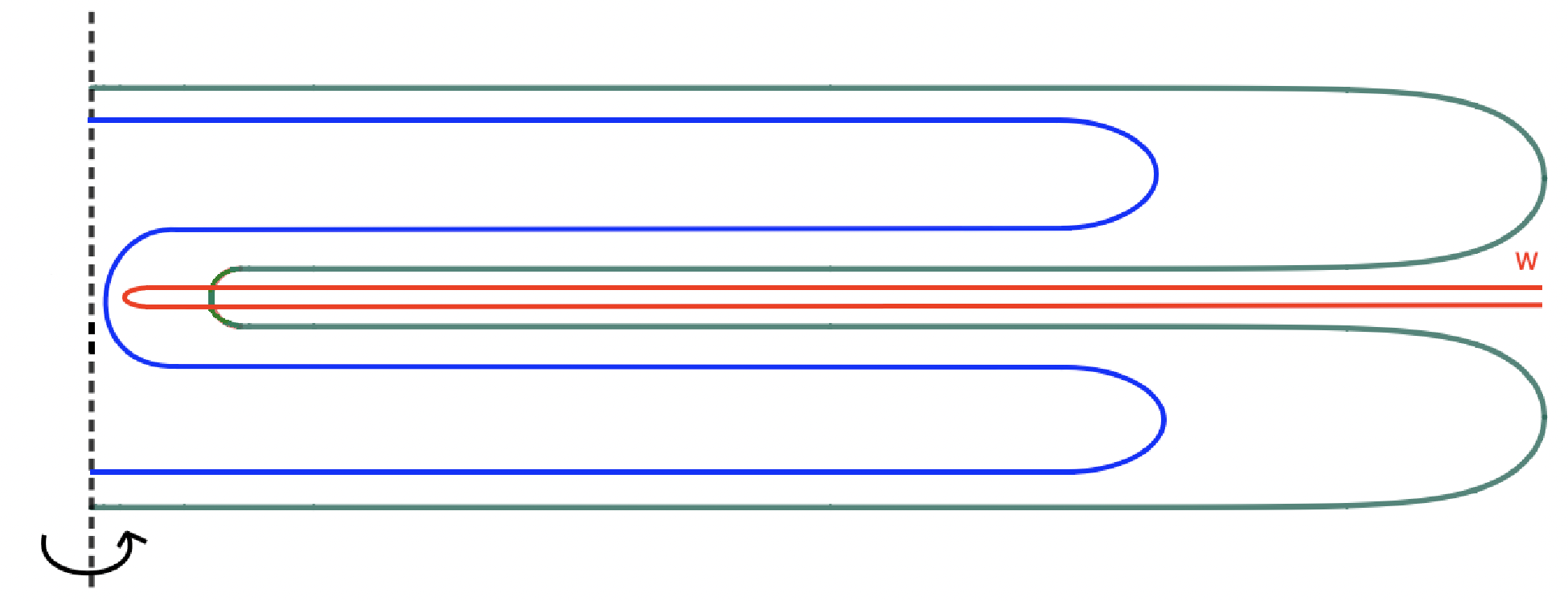}
    \caption{In the sketch above the red curve denotes the profile curve of the catenoid $W$, which we may arrange to intersect each $\left(\Gamma^{\overline{m}^0\left(i\right)}_{s_i}\right)_0$ (given in green) exactly twice. The blue curve denotes the hypothetical configuration where $m_i\left(t\right) < m_W$, so that $W$ later serves as a barrier.}
    \label{noaxis_sing}
\end{figure}

\begin{rem}\label{dimrem} In the argument above, the assumption that $n \geq 3$ entered when using an intersection counting argument with a catenoid, which in these dimensions lie in slabs. In dimension $n = 2$, the catenoid is entire and as a consequence, one can concoct plausible, from the point of view of the Sturmian principle, situations where the second part of the argument presented above, lower bounding $m_i(t)$, does not seem to hold as discussed in Figure \ref{lowdim_pic}. The scenario that $m(t)$ becomes small and then rebounds should be a possibility in a related context in \cite{SunChen_mult2plane}, where using continuous dependence of the flow appropriate perturbations of their construction should produce rotationally symmetric flows which "nearly" converge to a plane with multiplicity 2 in some neighborhood of the axis of rotation and then later on clear out. 
\end{rem}
\begin{figure}
    \centering
    \includegraphics[width=1\linewidth]{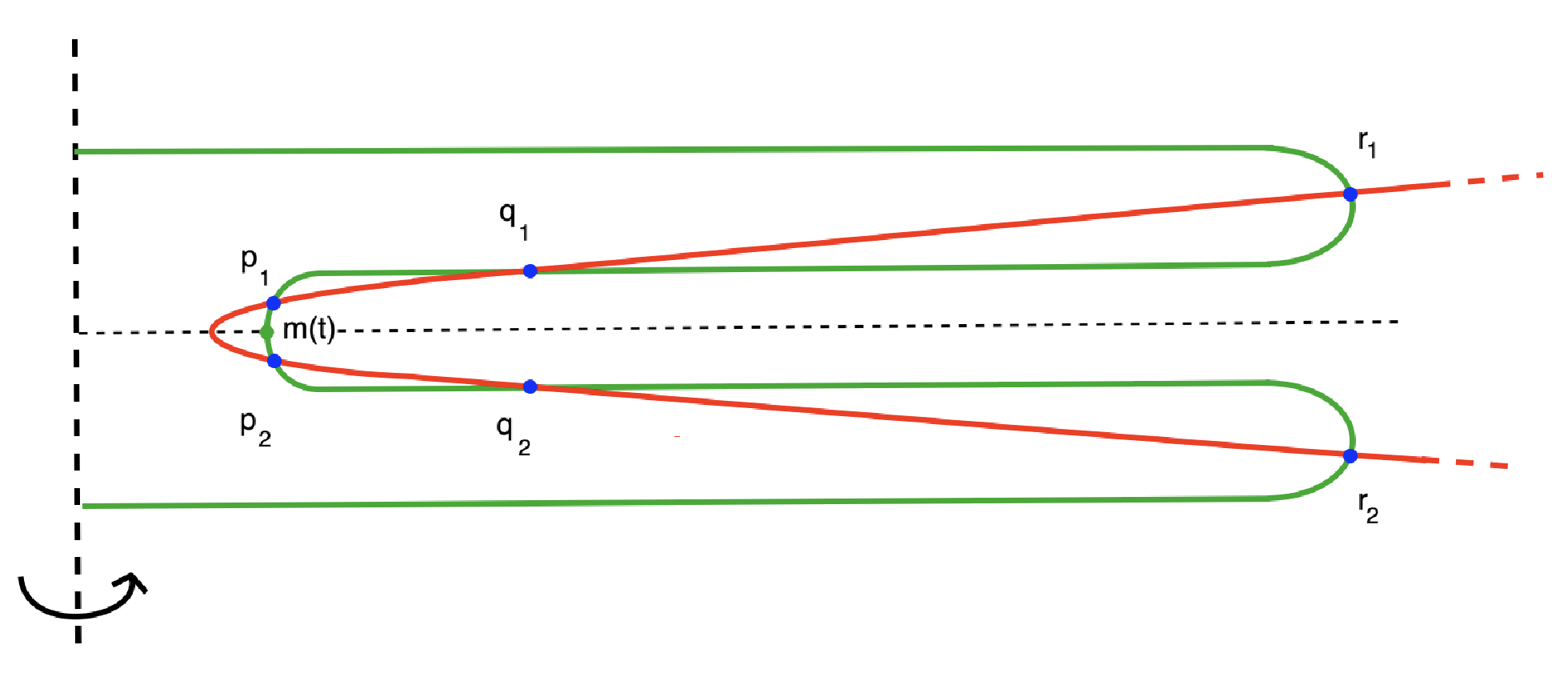}
    \caption{When $n = 2$ the profile curve of a fixed catenoid (with suitably thin neck), schematically illustrated in red, should intersect $\left(\Gamma^{\overline{m}_0(i)}_{s_i}\right)_0$, colored in green, in 6 points for $|s_i|$ large since the catenoid is entire in this dimension. These are labeled above by the pairs $p_j, q_j,$ and $r_j$. For $m(t)$ to become less than the neck width of the catenoid in the future, a plausible scenario is where the intersection points $p_1(t), p_2(t)$ come together to ``cancel'' and disappear (in a tangential intersection), leaving only $4$ intersection points remaining corresponding to the $q_j$ and $r_j$. In this situation $m(t)$ may yet become large again (i.e. $m(t) > 1/2$) by the stopping time by the points $q_1(t), q_2(t)$ canceling about the neck of the catenoid, at least without apparently violating the Sturmian principle.}\label{lowdim_pic}.
\end{figure}

        As a more hands on approach with dealing with off-axis singularities above it's not too hard to show that tangent flows at such points on $\Gamma^\infty_{t^\ast}$ must be modeled on multiplicity 2 or 4 lines.  With Lemma \ref{silky_smooth} in hand these correspond in $\Gamma^\infty_{t^\ast}$ to ``spikes'' terminating in ``tips,'' where the tips are limits of critical points on the approximating flows $\Gamma^i_t$. For $i \gg 1$ the curvature of $\Gamma^i_t$ about these points becomes increasingly large, so that they must push in or out (depending on if they are maxima or minima) very quickly. By an argument using the Sturmian principle one can show such tips on $\Gamma^\infty_{t^\ast}$ must immediately smooth out, giving a sequence of times $t_i \to -\infty$ for which the flow is smooth. This in fact suffices, by a uniqueness argument and Theorem \ref{aag_thm3}, to see that $B_t$ is a smooth flow. At any rate with this last statement our construction is essentially complete, as we discuss below:
        
        \begin{proof}[Proof of Theorem \ref{mainthm}] By Corollary \ref{ancient_flow_existence}, there exists an ancient rotationally and reflection symmetric Brakke flow $B_t$ such that $\supp(B_0)$ satisfies the desired properties: nonempty, compact and connected, and nonconvex (in the sense of not bounding a convex domain). By Lemma \ref{silky_smooth} we observed that if $\supp(B_t)$ is noncompact for some $t < 0$ it must contain what were called cusps or prongs which could be ruled out by a Sturmian principle argument applied to the approximating flows, giving that $\supp(B_t)$ is bounded for all negative times. Along with ruling out on-axis singularities, the boundedness gave us the necessary $C^0$ control to use Theorem \ref{aag_grad} to see that $B_{t}$ is also smoothly embedded for all times $t < 0$ in the proof of Lemma \ref{smooth_seq} (since it is smooth we now conflate it with $\supp(B_t)$). Of course $B_t$ is trapped in a slab for all $t < 0$ since the approximating flows are, so this all gives item (1) of the theorem holds. Since convexity is a preserved pinching condition along the flow we see that $B_t$ is nonconvex for all $t < 0$. Symmetries of the approximating flows are preserved in the limit and $B_t$ is embedded as discussed above, so item (2) holds. Having that the profile curves $\Gamma^i_t$ of the approximating flows are all graphical over the axis of rotation one can see, along with a Sturmian principle argument using planes perpendicular to the axis of rotation, that $\Gamma^\infty_t$ is itself graphical over the axis of rotation; the number of critical points can't increase or change (non-strict) type in the limit giving item (3) of the theorem by the established non-convexity.    \end{proof}

\bibliographystyle{plain}
\bibliography{bibliography}

\begin{thebibliography}{10}

\bibitem{aag}
Steven Altschuler, Sigurd~B. Angenent, and Yoshikazu Giga.
\newblock Mean curvature flow through singularities for surfaces of rotation.
\newblock {\em J. Geom. Anal.}, 5(3):293--358, 1995.

\bibitem{BK}
Richard Bamler and Bruce Kleiner.
\newblock On the multiplicity one conjecture for mean curvature flows of
  surfaces.
\newblock Preprint, arXiv:2312.02106.

\bibitem{blt_pancakes}
Theodora Bourni, Mat Langford, and Giuseppe Tinaglia.
\newblock Collapsing ancient solutions of mean curvature flow.
\newblock {\em J. Differential Geom.}, 119(2):187--219, 2021.

\bibitem{blt_poly}
Theodora Bourni, Mat Langford, and Giuseppe Tinaglia.
\newblock Ancient mean curvature flows out of polytopes.
\newblock {\em Geom. Topol.}, 26(4):1849--1905, 2022.

\bibitem{blm}
Theodora Bourni, Mathew Langford, and Alexander Mramor.
\newblock On the construction of closed nonconvex nonsoliton ancient mean
  curvature flows.
\newblock {\em Int. Math. Res. Not. IMRN}, (1):757--768, 2021.

\bibitem{BR}
Theodora Bourni and Martin Reiris.
\newblock Compact curve shortening flow solutions out of non-compact curves.
\newblock {\em Int. Math. Res. Not. IMRN}, (1):804--816, 2024.

\bibitem{BreKap}
Simon Brendle and Nikolaos Kapouleas.
\newblock Gluing {E}guchi-{H}anson metrics and a question of {P}age.
\newblock {\em Comm. Pure Appl. Math.}, 70(7):1366--1401, 2017.

\bibitem{chen2007uniqueness}
Bing-Long Chen and Le~Yin.
\newblock Uniqueness and pseudolocality theorems of the mean curvature flow.
\newblock {\em Communications in analysis and geometry}, 15(3):435--490, 2007.

\bibitem{SunChen_mult2plane}
Jingwen Chen and Ao~Sun.
\newblock Mean curvature flow with multiplicity 2 convergence in
  $\mathbb{R}^3$.
\newblock {\em Analysis \& PDE}, 19:1029--1060, 05 2026.

\bibitem{chini_moller}
Francesco Chini and Niels~Martin M{\o}ller.
\newblock Ancient mean curvature flows and their spacetime tracks.
\newblock Preprint, arXiv:1901.05481v2.

\bibitem{CCMS}
Kyeongsu Choi, Otis Chodosh, Christos Mantoulidis, and Felix Schulze.
\newblock Mean curvature flow with generic initial data.
\newblock {\em Invent. Math}, 237(1):121–220, 2024.

\bibitem{Chu2023GenusoneSI}
Adrian Chun~Pong Chu and Ao~Sun.
\newblock Genus-one singularities in mean curvature flow.
\newblock {\em Geometry \& Topology}, 2023.

\bibitem{CMentropy}
Tobias Colding and William Minicozzi.
\newblock Generic mean curvature flow i; generic singularities.
\newblock {\em Annals of Mathematics. Second Series}, 2, 03 2012.

\bibitem{EH}
Klaus Ecker and Gerhard Huisken.
\newblock Interior estimates for hypersurfaces moving by mean curvature.
\newblock {\em Invent. Math.}, 105(3):547--569, 1991.

\bibitem{HW}
Or~Hershkovits and Brian White.
\newblock Nonfattening of mean curvature flow of mean convex type.
\newblock {\em Comm. Pure Appl. Math.}, 73(3):558–580, 2020.

\bibitem{I2}
Tom Ilmanen.
\newblock Generalized flow of sets by mean curvature on a manifold.
\newblock {\em Indiana Univ. Math. J.}, 41(3):671--705, 1992.

\bibitem{I1}
Tom Ilmanen.
\newblock Elliptic regularization and partial regularity for motion by mean
  curvature.
\newblock {\em Mem. Amer. Math. Soc.}, 108(520):x+90, 1994.

\bibitem{LP}
Tang-Kai Lee and Alec Payne.
\newblock An intersection principle for mean curvature flow.
\newblock Preprint, arXiv:2505.11600.

\bibitem{ma_immersed_halfspace}
John Man~Shun Ma.
\newblock Ancient solutions to the curve shortening flow spanning the
  halfplane.
\newblock {\em Trans. Amer. Math. Soc.}, 374(6):4207--4226, 2021.

\bibitem{mant}
Carlo Mantegazza.
\newblock {\em Lecture notes on mean curvature flow}, volume 290 of {\em
  Progress in Mathematics}.
\newblock Birkh\"auser/Springer Basel AG, Basel, 2011.

\bibitem{Mra1}
Alexander Mramor.
\newblock A classification result for eternal mean convex flows of finite total
  curvature type.
\newblock Preprint, arXiv:2403.12020.

\bibitem{MP}
Alexander Mramor and Alec Payne.
\newblock Ancient and eternal solutions to mean curvature flow from minimal
  surfaces.
\newblock {\em Math. Ann.}, 380(1-2):569--591, 2021.

\bibitem{MP1}
Alexander Mramor and Alec Payne.
\newblock Nonconvex surfaces which flow to round points.
\newblock {\em Comm. Anal. Geom.}, 32(3):837--887, 2024.

\bibitem{schulze_grenoble_notes}
Felix Schulze.
\newblock An introduction to brakke flows.
\newblock Lecture Notes. Summer School 2021: Curvature Constraints and Spaces
  of Metrics. Institut Fourier, Grenoble.

\bibitem{Ton}
Yoshihiro Tonegawa.
\newblock {\em Brakke's mean curvature flow}.
\newblock SpringerBriefs in Mathematics. Springer, Singapore, 2019.
\newblock An introduction.

\bibitem{traizet}
Martin Traizet.
\newblock An embedded minimal surface with no symmetries.
\newblock {\em J. Differential Geom.}, 60(1):103--153, 2002.

\bibitem{White_avoidance}
Brian White.
\newblock The avoidance principle for noncompact hypersurfaces moving by mean
  curvature flow.
\newblock {\em Calc. Var. Partial Differential Equations}, 63(5):Paper No. 111,
  20, 2024.

\end{thebibliography}

\end{document}